\documentclass[12pt,reqno]{amsart}
\usepackage{amsmath,amssymb,amsfonts,amsthm,amscd,indentfirst}
\usepackage[a4paper, left=2.5cm, right=2.5cm, top=2cm, bottom=2cm]{geometry}
\usepackage{hyperref}
\usepackage[numbers,sort]{natbib}

\hypersetup{
colorlinks=true,
linkcolor=black
}

\numberwithin{equation}{section}

\newtheorem{proposition}{Proposition}[section]
\newtheorem{theorem}[proposition]{Theorem}
\newtheorem{lemma}[proposition]{Lemma}

\newtheorem{remark}[proposition]{Remark}

\newtheorem{definition}[proposition]{Definition}

\newcommand{\ttt}{\theta}
\newcommand{\lam}{\lambda}

\newcommand{\xg}{\backslash}

\newcommand{\wq}{\infty}
\newcommand{\te}{\text}

\newcommand{\rr}{\mathbb{R}}

\newcommand{\pa}{\partial}

\newcommand{\p}{\phi}

\newcommand{\pp}{\varphi}
\newcommand{\PP}{\varPhi}

\newcommand{\PPP}{\varPsi}
\newcommand{\es}{\varepsilon}

\newcommand{\dd}{\Delta}
\newcommand{\de}{\delta}
\newcommand{\na}{\nabla}
\newcommand{\al}{\alpha}
\newcommand{\bb}{\beta}
\newcommand{\ga}{\gamma}
\newcommand{\gaa}{\Gamma}

\newcommand{\ooo}{\Omega}
\newcommand{\si}{\sigma}

\newcommand{\z}{\left}
\newcommand{\x}{\right}

\newcommand{\ka}{\kappa}

\allowdisplaybreaks[4]
 
\usepackage{graphicx}

\begin{document}
\title[The Neumann problem for sum Hessian equations]{Neumann Problems for Elliptic and Parabolic Sum Hessian Equations}
\author[Weizhao Liang]{Weizhao Liang}
\address{School of Mathematical Sciences, University of Science and Technology of China, Hefei, China.}
\email{Lwz740@mail.ustc.edu.cn}
\author[Jin Yan]{Jin Yan}
\address{School of Mathematical Sciences, University of Science and Technology of China, Hefei, China.}
\email{yjoracle@mail.ustc.edu.cn}
\author[Hua Zhu]{Hua Zhu}
\address{School of Mathematical and Physics, Southwest University of Science and Technology, Mianyang, 621010, Sichuan Province, China.}
\email{zhuhmaths@mail.ustc.edu.cn}

\begin{abstract}
    This paper studies the Neumann boundary value problem for sum Hessian equations. We first derive a priori $C^2$ estimates for $(k-1)$-admissible solutions in almost convex and uniformly $(k-1)$-convex domains, and prove the existence of admissible solutions via the method of continuity. Furthermore, we obtain existence results for the classical Neumann problem in uniformly convex domains. Finally, we extend these results to the corresponding parabolic problems.
\end{abstract}


\maketitle

\section{Introduction}\label{sec 1-intro}
\setcounter{equation}{0}

In this paper, we study the following Neumann problem for the sum Hessian equation:
\begin{equation}\label{eq 1.1}
    \left\{\begin{aligned}
    &\sigma_{k}(\lam(D^2u))+\al \sigma_{k-1}(\lam(D^2u))  =f(x) &&  \text {in } \Omega, \\
    &u_\nu  =\pp(x,u) && \text {on } \pa \ooo,
    \end{aligned}\right.
\end{equation}
where  $n\geq k\geq 2$, $\al>0$, and $\ooo\subset \rr^n$ is a bounded domain. Here, $f$ is a positive function in $\overline{\ooo}$, and $\nu$ denotes the outer unit normal vector of $\pa \ooo$. For any $1\leq k\leq n$, the $k$-th elementary symmetric polynomial of the eigenvalues $\lam(D^2 u)=(\lam_1,\cdots, \lam_n)$ of $D^2 u$ is given by
$$
\si_k(D^2 u) = \si_k(\lam(D^2 u))=\sum_{1\leq i_1<\cdots<i_k\leq n} \lam_{i_1} \cdots \lam_{i_k}.
$$
The $k$-Hessian equation $\sigma_k(D^2u)=f(x,u,Du)$ is fundamental in fully nonlinear PDE theory and geometric analysis, with deep connections to various geometric problems (see, e.g., \cite{BK,CNS1,CNS2,CNS3,CNS4,CY,GLL,GLM,O,P3,TW}). 

The Neumann and oblique derivative problems for partial differential equations have been extensively studied in the literature.
When $\al=0$, equation \eqref{eq 1.1} reduces to the Neumann problem for the $k$-Hessian equation. The case $k=1$ (the Laplace equation) is classical, and we refer to \cite{GT} for a priori estimates and the existence results. For the Monge-Amp\` ere case ($k=n$), Lions et al. \cite{LTU} first established the solvability of the Neumann problem. The intermediate cases $2\leq k\leq n-1$ present more complex challenges. Trudinger \cite{Tru-1987} solved this problem when the domain is a ball and he conjectured that similar results should hold for general uniformly convex domains. This conjecture was later confirmed by Ma and Qiu \cite{Ma-Qiu-2019-CPM}. Further developments include  the work of Chen and Zhang \cite{Chen-Zhang-Bull-2021}, who solved the Neumann problem for Hessian quotient equations. 

Chen et al. \cite{Chen-Chen-Mei-Xiang-2022-real-mixed} studied the mixed Hessian equation $\si_{k}(D^2 u) +\sum_{l=1}^{k-1} \al_l(x)\si_l(D^2 u)=f(x)$. For another important mixed Hessian equation, the special Lagrangian-type equation $\sum_i \arctan\frac{\lam_i(D^2 u)}{f(x)}=\Theta$, Chen et al. \cite{chen-ma-wei-2019-real-JDE,chen-ma-wei-2019-complex} considered the case $f\equiv 1$, with the general case $f>0$ later treated in \cite{Qiu-Zhang-arXiv-2024}.
We should emphasize that the mixed Hessian equation in \cite{Chen-Chen-Mei-Xiang-2022-real-mixed}, with coefficient satisfying $\al_{k-1}\leq 0$ and $\al_l<0$ for $1 \leq l\leq k-2$, presents a fundamentally different structure from our equation \eqref{eq 1.1}, where the key parameter $\al$ is positive. Moreover, their approach converts the equation to a Hessian quotient form by division with $\si_{k-1}$, and hence can not deal with the case $f=f(x,u)$ in second-order derivative estimates.

For the classical Neumann problem in a uniformly convex domain, Lions et al. \cite{LTU} studied the case $k=n$, while Qiu and Xia \cite{Qiu-Xia-IMRN-2019} generalized these results to arbitrary $k$-Hessian equations. More results can be found in \cite{Chen-Chen-Mei-Xiang-2022-real-mixed,Chen-Chen-Mei-Xiang-2022-complex-mixed,chen-ma-zhang-2021,Chen-Zhang-Bull-2021}. 

The study of oblique derivative problems for $k$-Hessian equations has seen significant developments through several key contributions. The special case when $k=n=2$ was resolved by Wang \cite{Wang-1992-Oblique}, while Urbas \cite{Urbas-1998} later established existence results for the Monge-Amp\` ere equations in general dimensions. Building on these foundations, Jiang and Trudinger in \cite{Jiang-Trudinger-I-2018,Jiang-Trudinger-II-2017,Jiang-Trudinger-III-2019} obtained existence theorems for certain augmented Hessian equations with oblique boundary conditions. In \cite{Wang Peihe-2022}, Wang derived global gradient estimates for admissible solutions to Hessian equations. For additional results on Neumann and oblique derivative problems for linear and quasilinear elliptic equations, we refer to Lieberman's books \cite{Lieberman-Book-2013}.

Equation \eqref{eq 1.1} provides a natural generalization of the $k$-Hessian equation through a linear combination of $k$-Hessian operators. This class of equations has profound geometric applications and has been widely studied in various contexts.  In particular, Harvey and Lawson \cite{HL80} studied such operators in their work on special Lagrangian equations. The concavity properties of these operators enabled Krylov \cite{Kry} and Dong \cite{Dong} to establish important curvature estimates. Guan and Zhang \cite{GZ} derived curvature estimates for related equations where the right-hand side does not depend on the gradient but involves coefficients tied to hypersurface positions. Additionally, certain problems in hyperbolic geometry can also be formulated as equations of this type \cite{EGM}.

We are particularly interested in the existence of admissible solutions to the Neumann problem \eqref{eq 1.1}. In this paper, we establish existence results for such solutions in almost convex and uniformly $(k-1)$-convex domains by deriving a priori $C^2$ estimates. Moreover, we extend these results to the classical Neumann problem for sum Hessian equations and their parabolic versions in uniformly convex domains.

We begin by defining $(k-1)$-admissible solutions and admissible sets for the sum Hessian equation, along with the required convexity conditions for the domain $\ooo$.

\begin{definition}\label{def 1.1}
    Let $\ooo \subset \mathbb{R}^n$ be a bounded $C^2$ domain and $u\in C^2(\ooo)$. 
    \par
    $(1)$ 
    A function $u$ is called $k$-convex if the eigenvalues of $D^2 u$ satisfy
    $$\lambda(D^2 u(x)) = (\lambda_1(x), \cdots, \lambda_n(x)) \in \gaa_k, \quad \te{for any } x\in \ooo,
    $$ 
    where $\Gamma_k$, the G\aa rding's cone, is defined as $\Gamma_k = \{ \lambda \in \mathbb{R}^n \mid \sigma_i(\lambda) > 0, \, i = 1, \dots, k \}$.
    \par
    $(2)$ 
    For fixed $\al >0$. A function $u$ is called $(k-1)$-admissible if $\lam(D^2 u(x)) \in \tilde{\gaa}_k$ for all $x\in\ooo$, where $\tilde{\gaa}_{k}$ is defined as
    \begin{equation*}
        \tilde{\Gamma}_k=\Gamma_{k-1}\cap\{\lambda \in \rr^n \mid \sigma_k(\lambda)+\alpha\sigma_{k-1}(\lambda)>0\}.
    \end{equation*}
    Clearly, $\gaa_{k}\subseteq \tilde{\gaa}_k\subseteq \gaa_{k-1}$. By Lemma \textup{\ref{lem 2.2}} $(a)$, $\tilde\Gamma_k$ is a convex set, and the operator $\si_k +\al \si_{k-1}$ is elliptic in $\tilde{\gaa}_k$. 
    \par
    $(3)$ 
    We say that $\ooo$ is almost convex if there exists a small negative constant $a_\ka$ such that 
    $$
    \ka_i\geq a_\ka,\quad\te{on } \pa \ooo,  \te{ for any }  1\leq i\leq n-1,
    $$ 
    where $\ka=(\ka_1,\cdots , \ka_{n-1})$ denotes the principal curvatures of $\pa \ooo$ with respect to its inner normal. Similarly,  
    $\ooo$ is uniformly $(k-1)$-convex if $\ka \in \gaa_{k-1}$ and $\si_{k-1}(\ka)\geq c_{\ka}$ on $\pa \ooo$ for some positive constant $c_{\ka}$.
    $\ooo$ is uniformly convex if $\min_{i}\ka_i\geq \ga_\ka$ on $\pa \ooo$ for some positive constant $\ga_{\ka}$.
\end{definition}

Throughout the paper, unless explicitly stated otherwise, $\ooo$ is a bounded domain and $\nu$ denotes the unit outer normal vector field of $\pa \ooo$. We now state our first main result.

\begin{theorem}\label{thm 1.2-Neumann for elliptic sum}
    Let $\ooo \subset \rr^n$ be a $C^4$ almost convex and uniformly $(k-1)$-convex domain, $f \in C^{2}(\overline{\ooo})$, $\pp \in C^{3}(\pa \ooo\times\rr)$ satisfy $f>0$, $\pp_u\leq c_\pp<0$, and $2a_{\ka}>c_{\pp}$. Then there exists a unique $(k-1)$-admissible solution $u\in C^{3,\ga}(\overline{\ooo})$ to the Neumann problem
    \begin{equation}\label{eq-Neumann for elliptic sum (thm 1.2)}
        \left\{\begin{aligned}
            & \si_k(D^2 u)+\al \si_{k-1}(D^2 u )=f(x) && \te{in } \ooo, \\
            & u_\nu =\pp(x,u) && \te{on } \pa \ooo, 
        \end{aligned}\right.
    \end{equation} 
    where $\ga\in (0,1)$ is a constant.
\end{theorem}
\begin{remark}
     It is well known that the Dirichlet problem for the k-Hessian equation was solved in uniformly $(k-1)$-convex domains, as demonstrated by Caffarelli et al. \textup{\cite{CNS3}}. 
    However, for the Neumann problem , the domain $\ooo$ is typically required to be uniformly convex, see Trudinger \textup{\cite{Tru-1987}} for balls, Ma and Qiu \textup{\cite{Ma-Qiu-2019-CPM}} for the uniformly convex doamins, and Chen and Zhang \textup{\cite{Chen-Zhang-Bull-2021}} for convex $($i.e., $\min_i \ka_i \geq 0$$)$ and uniformly $(k-1)$-convex assumptions. We should mention that our almost convexity assumption is just a trivial consequence in the calculation, and the condition is far from being relaxed to uniform $(k-1)$- convexity. 
\end{remark}

Following the approach of \cite{Chen-Zhang-Bull-2021,LTU,Qiu-Xia-IMRN-2019}, we establish an existence theorem for the classical Neumann problem of sum Hessian equations.
\begin{theorem}\label{thm 1.3-classical Neumann for elliptic sum}
    Let $\ooo \subset \rr^n$ be a $C^4$ uniformly convex domain, $f\in C^{2}(\overline{\ooo} )$ satisfy $f>0$ and $\pp \in C^{3}(\pa \ooo)$. Then there exists a unique constant $s$ and a $(k-1)$-admissible solution $u\in C^{3,\ga}(\overline{\ooo})$ solving the classical Neumann problem
    \begin{equation}\label{eq-classical Neumann for elliptic sum (thm 1.3)}
        \left\{\begin{aligned}
            & \si_k(D^2 u)+\al \si_{k-1}(D^2 u )=f(x) && \te{in } \ooo, \\
            & u_\nu =s+\pp(x) && \te{on } \pa \ooo. 
        \end{aligned}\right.
    \end{equation}
    The function $u$ is unique up to a constant.
\end{theorem}
\begin{remark}
    Note that the classical Neumann problem \eqref{eq-classical Neumann for elliptic sum (thm 1.3)} lacks uniform $C^0$ estimate for $u$. Following \textup{\cite{Chen-Zhang-Bull-2021,LTU,Qiu-Xia-IMRN-2019}}, we instead study the $(k-1)$-admissible solution $u_\es$ of the perturbed Neumann problem 
    \begin{equation}\label{eq elliptic perturbed Neumann problem}
        \left\{\begin{aligned}
            & \si_k(D^2 u_\es)+\al \si_{k-1}(D^2 u_\es )=f(x) && \te{in } \ooo, \\
            & (u_\es)_\nu =-\es u_\es +\pp(x) && \te{on } \pa \ooo, 
        \end{aligned}\right.
    \end{equation}
    for small $\es >0$. To derive uniform gradient and second-order derivatives estimate for $u_\es$ that are independent of $\es$, we require $\ooo$ to be uniformly convex. Taking $\es \to 0$, we can obtain a solution to \eqref{eq-classical Neumann for elliptic sum (thm 1.3)}, with uniqueness following immediately from the maximum principle and Hopf's lemma.
\end{remark}

For Neumann problem of parabolic sum Hessian equations. Previous work has addressed several cases, including the Monge-Amp\`ere equation \cite{Schnurer-Schwetlick-PJM-2004}, mean curvature equation \cite{Ma-Wang-Wei-2018-JFA}, and Hessian quotient equation \cite{chen-ma-zhang-2021}. We now extend this result to sum Hessian case:
\begin{equation}\label{eq parabolic sum hessian-1}
    \left\{\begin{aligned}
        & u_t = \log S_k(D^2 u)-\log f(x,u) && \te{in } \ooo \times [0,T),\\
        & u_\nu=\pp(x,u) && \te{on } \pa \ooo\times [0,T),\\
        & u(x,0)=u_0(x) && \te{in }  \ooo, 
    \end{aligned}\right.
\end{equation} 
where $u$ is a function defined on $\overline{\ooo}\times[0,T)$, $S_k$ is defined in Definition \ref{def 2.1}, and $T$ is the maximal time.
Before stating our theorem, we first introduce some structural conditions on $f$, $\pp$ and $u_0$. We assume $f$, $\pp$, and $u_0$ are smooth functions (in parabolic problems) and  
\begin{align}
    & f>0 \quad \te{and} \quad f_u\geq  0, \label{condition f>0} \\
    & \pp_u\leq c_\pp<0, \label{condition c_pp} \\
    & \frac{f_u}{f}\geq c_f>0,\label{condition c_f}\\
    & S_k(D^2 u_0)\geq f(x,u_0),\label{condition S_k(D^2 u_0)} \\
    & \pa_t^l (u_\nu-\pp(x,u))|_{t=0}=0, \quad \forall \  l\geq 0, \ x\in \pa \ooo.\label{condition compatibility}
\end{align}

Our third main theorem is
\begin{theorem}\label{thm 1.5-parabolic}
    Let $\ooo \subset \rr^n$ be a smooth, almost convex and uniformly $(k-1)$-convex domain, $f,\pp\in C^\wq(\overline{\ooo}\times \rr)$ satisfy \eqref{condition f>0}, \eqref{condition c_pp} and $2a_\ka >c_{\pp}$. Furthermore, assume that one of the conditions \eqref{condition c_f} and \eqref{condition S_k(D^2 u_0)} holds. If there exists a smooth $(k-1)$-admissible function $u_0$ satisfying the compatibility condition \eqref{condition compatibility}. Then, equation \eqref{eq parabolic sum hessian-1} admits a unique smooth solution $u(x,t)$ for all $t\geq 0$. Moreover, as $t\to\wq$, $u(x,t)$ converges smoothly to a unique solution $u^\wq$ of the Neumann problem
    \begin{equation}\label{eq convergent elliptic Equation}
        \left\{\begin{aligned}
            & S_k(D^2 u)=f(x,u) && \te{in } \ooo, \\
            & u_\nu=\pp(x,u) &&\te{on } \pa \ooo.
        \end{aligned}\right.
    \end{equation}
    In particular, the convergence is exponential if \eqref{condition c_f} holds.
\end{theorem}

\begin{theorem}\label{thm 1.6-classical parabolic}
    Let $\ooo \subset \rr^n$ be a smooth uniformly convex domain, $f,\pp \in C^\wq (\overline{\ooo})$ satisfy $f>0$. If $u_0\in C^\wq(\overline{\ooo})$ is a $(k-1)$-admissible function satisfying
    $(u_0)_\nu=\pp(x)$ on $\pa \ooo$.
    Then there exists a unique smooth $(k-1)$-admissible solution $u(x,t)$ to 
    \begin{equation}\label{eq parabolic sum hessian-2}
        \left\{\begin{aligned}
            & u_t = \log S_k(D^2 u)-\log f(x) && (x,t)\in \ooo \times (0,T),\\
            & u_\nu=\pp(x) &&  (x,t)\in \pa \ooo\times [0,T),\\
            & u(x,0)=u_0(x) &&  x\in  \overline\ooo, 
        \end{aligned}\right.
    \end{equation}
    where $u(\cdot,t)$ converges to $u_0$ in $C^2(\overline{\ooo})$ as $t\to 0$. Moreover, $u(\cdot,t)$ converges smoothly to a translating solution $s^\wq t+u_{\operatorname{ell}}^\wq$, where $(s^\wq,u_{\operatorname{ell}}^\wq)$ is a suitable solution to \eqref{eq 7.1}.
\end{theorem}

This paper is organized as follows. Section \ref{sec 2-pre} introduces the necessary notation and establishs several key inequalities of sum Hessian operators. Section \ref{sec 3-C^0 and C^1} is devoted to deriving the $C^0$ and global gradient estimates, which are inspired by the works of \cite{Chou-Wang-CPAM-2001} and \cite{Wang Peihe-2022}. In Section \ref{sec 4-second derivative}, we derive the global second-order derivatives estimate. The proofs of Theorems \ref{thm 1.2-Neumann for elliptic sum} and \ref{thm 1.3-classical Neumann for elliptic sum} are presented in Section \ref{sec-existence}. Finally, Sections \ref{sec-parabolic} and \ref{sec-classical parabolic} investigate the Neumann problem for the parabolic versions, proving Theorems \ref{thm 1.5-parabolic} and \ref{thm 1.6-classical parabolic} respectively.

\section{Preliminaries}\label{sec 2-pre}
\setcounter{equation}{0}

In this section, we introduce some notations which are the same as \cite{Liang-Yan-Zhu-2024-arXiv,Liu-Ren-2023-JFA}, then establish some inequalities for the later use. For convenience, we adopt the following convention for constants: constants depending only on $n$ and $k$ will be denoted by $c_{n,k}$, $c(n,k)$, $C_{n,k}$, or $C(n,k)$, while constants that additionally depend on the information of the equation will be denoted by $C_1,C_2,\cdots$. As usual, these constants may change from line to line.

\begin{definition}\label{def 2.1}
    Let $\lam=(\lam_1,\cdots,\lam_n)\in \rr^n$. For any $1\leq i,j \leq n$, we have
    \begin{enumerate}
        \item For any fixed $\al >0$. We define 
        $$
        S_k(\lam):=\si_k(\lam)+\al \si_{k-1}(\lam) \quad \te{and} \quad  S_{k,i}(\lam)=S_{k}(\lam|i):=S_k(\lam)\big|_{\lam_i=0}.
        $$
        \item $S_k(\lam)=\lam_i S_{k-1}(\lam|i)+S_k(\lam|i)$.
        \item $S_k^{ii}(\lam):=\dfrac{\pa S_k(\lam)}{\pa \lam_i}=\si_{k-1}(\lam|i)+\al\si_{k-2}(\lam|i)=S_{k-1}(\lam|i)$.
        \item $S_k^{ii,jj}(\lam):=\dfrac{\pa S_k(\lam)}{\pa \lam_i \pa \lam_j}=S_{k-2}(\lam|ij)=S_{k-2,ij}(\lam)$. In particular, $S_k^{ii,ii}(\lam)=0$.
        \item $\sum\limits_{i=1}^n S_k^{ii}(\lam)=(n-k+1)\si_{k-1}(\lam)+\al(n-k+2)\si_{k-2}(\lam)$. 
        \item $\sum\limits_{i=1}^n S_k(\lam|i)=(n-k)S_{k}(\lam)+\al\si_{k-1}(\lam)$. 
        \item $\sum\limits_{i=1}^{n}\lam_i S_{k-1}(\lam|i)=kS_k(\lam)-\al\si_{k-1}(\lam)$.
    \end{enumerate}
\end{definition}

The following lemmas are very useful for us to establish a priori estimates.
\begin{lemma}\label{lem 2.2}
    Let $\lambda=(\lambda_1,\cdots,\lambda_n)\in\tilde\Gamma_k$, $1\leq k\leq n$, and  $\lambda_1 \geq \lam_2\geq\cdots \geq\lambda_n$, then
    \par
    $(a)$ $\tilde{\gaa}_{k}$ is a convex set and $\si_{k}(\lam)+\al \si_{k-1}(\lam)$ is elliptic in $\tilde{\gaa}_k$.
    \par
    $(b)$ $S_k^{\frac{1}{k}}(\lam)$ and $\z(\frac{S_k}{S_l}\x)^{\frac{1}{k-l}}(\lam)$ are concave in $\tilde\gaa_k$ for $1\leq l+1 \leq  k\leq n$.
    \par
    $(c)$ There exists a positive constant $\theta_0$ depending on $n,k$, such that
    \begin{equation}\label{eq 2.1}
        \lam_i S_k^{ii}(\lam)\geq \theta_0 S_k(\lam), \quad \te{for any } 1\leq i \leq k-1.
    \end{equation}
\end{lemma}
\begin{proof}
    The proofs of $(a)$ and $(b)$ can be found in \cite[Theorem 10 and Corollary 13]{LRW1}, and $(c)$ is proved in \cite[Lemma 2.4]{Liu-Ren-2023-JFA}. For brevity, we omit the details here.
\end{proof}

\begin{lemma}\label{lem 2.3}
    Let $\lam=(\lam_1,\cdots,\lam_n)\in\tilde\Gamma_k$, $1\leq k\leq n$. Assume that $S_k\geq \tau_0$ for some positive constant $\tau_0$. Then, there exists $c_0>0$ depending on $n$, $k$ and $\tau_0$, such that 
    $$
    \si_{k-1}(\lam)\geq c_0, \quad \te{and} \quad \sum_{i=1}^n S_k^{ii}(\lam)\geq c_0.
    $$
\end{lemma}
\begin{proof}
    If $\si_k>0$, that is, $\lam \in \gaa_k$, then by the Newton-MacLaurin inequality, we have
    $$      
    \si_{k-1}\geq c(n,k) \si_{k}^{\frac{k-1}{k}}. 
    $$
    This implies 
    $$
    \tau_0 \leq S_k \leq \z(\frac{\si_{k-1}}{c(n,k)}\x)^{\frac{k}{k-1}}+\al \si_{k-1}\leq C(n,k) \max\z\{\si_{k-1}^{\frac{k}{k-1}},\si_{k-1}\x\}.
    $$
    If $\si_k\leq 0$, then $\tau_0\leq S_k \leq \al \si_{k-1}$. In both cases, we obtain
    $$
    \si_{k-1}\geq c_0(n,k,\tau_0)>0.
    $$
    \par
    From the definition of $S_k^{ii}$, it follows that
    $$
    \sum_{i} S_k^{ii}=(n-k+1)\si_{k-1}+\al (n-k+2)\si_{k-2} \geq c_0(n,k,\tau_0)>0.
    $$
\end{proof}

\begin{lemma}\label{lem 2.4}
    Let $\lam=(\lam_1,\cdots, \lam_n)\in \tilde{\gaa}_k$, $1\leq k\leq n$, and $\lam_1<0$. Then, we have 
    $$
    S_k^{11}(\lam)\geq \frac{1}{n-k+2}\sum_{i=1}^n S_k^{ii}(\lam):=c_{n,k}\sum_{i=1}^n S_k^{ii}(\lam).
    $$
\end{lemma}
\begin{proof}
    Since $\lam_1<0$, then $\si_i(\lam)=\lam_1 \si_{i-1} (\lam|1)+\si_i(\lam|1 )<\si_i(\lam|1)$ for $1\leq i\leq k-1$ and 
    \begin{align*}
        S_k^{11}& =\si_{k-1}(\lam|1)+\al \si_{k-2}(\lam|1)> \si_{k-1}+\al \si_{k-2} \\
        & \geq \frac{1}{n-k+2}\z[(n-k+1)\si_{k-1}+\al (n-k+2)\si_{k-2}\x]\\
        & = \frac{1}{n-k+2} \sum_{i} S_k^{ii}.
    \end{align*}
\end{proof}

Inspired by \cite[Lemma 2.7]{Chen-Zhang-Bull-2021}, we have the following lemma.
\begin{lemma}\label{lem 2.5}
    Let  $\lam=(\lam_1,\cdots,\lam_n)\in \tilde{\gaa}_k$, $k\geq 2$, and $\lam_2\geq \cdots \geq \lam_n$. For the following two cases: $(1)$ $n\geq k\geq 3$ and $\lam_1>0$, or $(2)$ $n>k=2$, $\lam_1\geq 1$ and $\lam_2> 0$. If, in addition,
    $$
    \lam_n<0, \quad \lam_1\geq \de \lam_2, \quad \te{and}\quad -\lam_n\geq \es \lam_1,
    $$ 
    for small positive constants $\de$ and $\es$, then we have
    \begin{equation}\label{eq S_k^11 geq theta_1 S_k^nn}
        S_k^{11}(\lam)\geq \ttt_1 S_k^{nn}(\lam),
    \end{equation}
    where $\ttt_1>0$ is a small constant depending on $n$, $k$, $\es$, and $\de$. 
\end{lemma}
\begin{proof}
    First, we claim that there exists a constant $c(n,k)>0$, such that 
    \begin{equation}\label{eq lam_1 S_k^11,nn geq c de S_k^nn}
        \lam_1 S_k^{11,nn}\geq  c(n,k) \de S_k^{nn}.
    \end{equation}
    
    For the case of $n>k=2$.
    \begin{align*}
        \de S_2^{nn}\leq S_2^{nn}=\sum_{i<n} \lam_i+\al \leq C \lam_1 =C \lam_1 S_2^{11,nn}. 
    \end{align*}
    For the case of $n\geq k\geq 3$. If $\lam_1\geq \lam_2$, by Lemma \ref{lem 2.2} $(c)$, we have
    \begin{align*}
        \lam_1 S_{k}^{11,nn}\geq \ttt_0 S_{k}^{nn}.
    \end{align*}
    If $\lam_1< \lam_2$, using again \eqref{eq 2.1},  we get 
    $$
    \lam_1 S_k^{11,nn}
    \geq \de \lam_2 S_k^{22,nn} \geq \de \ttt_0 S_k^{nn}.
    $$
    
    In the following, we divide into two cases to prove \eqref{eq S_k^11 geq theta_1 S_k^nn}. Denote $\ttt=\frac{\es \de (k-1)}{2k(n-2)}$.

    \textbf{Case 1:} $S_k^{11}\geq \ttt (-\lam_n )S_k^{11,nn}$. In this case, we have
    $$
    S_k^{11}\geq \ttt(-\lam_n)S_k^{11,nn}\geq \ttt \es \lam_1 S_k^{11,nn} \geq \es \ttt  \de c(n,k) S_k^{nn}.    
    $$

    \textbf{Case 2:} $S_k^{11}<\ttt(-\lam_n) S_k^{11,nn}$. In this case, we divide into two subcases to get the result.

    \textbf{Case 2a:} $\si_{k-1,1}> 0$. Then $(\lam|1)\in \gaa_{k-1}$ and 
    \begin{align*}
        S_{k,1} =& \si_{k,1}+\al \si_{k-1,1 } \\
        =& \frac{1}{k}\sum_{i>1} \lam_i \si_{k-1,1i}+\frac{\al}{k-1}\sum_{i>1} \lam_i \si_{k-2,1i}\\
        =& \frac{1}{k}\sum_{i>1} \lam_i \z[\si_{k-1,1}-\lam_i \si_{k-2,1i}\x]+\frac{\al}{k-1}\sum_{i>1} \lam_i \z[\si_{k-2,1}-\lam_i \si_{k-3,1i}\x] \\
        =& \sum_{i>1} \lam_i \z[\frac{1}{k}\si_{k-1,1}+\frac{\al}{k-1}\si_{k-2,1}\x] -\sum_{i>1} \lam_i^2 \z[\frac{1}{k}\si_{k-2,1i}+\frac{\al }{k-1}\si_{k-3,1i}\x] \\
        \leq &  \frac{(n-2)\lam_2}{k-1} \z[\si_{k-1,1}+\al \si_{k-2,1}\x] -\frac{\lam_n^2}{k} \z[\si_{k-2,1n}+\al \si_{k-3,1n}\x] \quad (\te{use } (\lam|1)\in \gaa_{k-1})  \\
        \leq & \frac{(n-2)\lam_1}{ (k-1) \de}S_k^{11}-\frac{\lam_n^2}{k}S_k^{11,nn} \\
        \leq & \frac{(n-2)\lam_1}{(k-1) \de} \ttt(-\lam_n) S_k^{11,nn} + \frac{\es \lam_1 \lam_n}{k}S_k^{11,nn}  \quad (\te{use Case 2}) \\
        = & -\frac{\es}{2k}\lam_1 (-\lam_n)S_k^{11,nn}<0.
    \end{align*}
    It follows from \eqref{eq lam_1 S_k^11,nn geq c de S_k^nn} that
    \begin{align}\label{eq 2.4}
        S_k^{11}=\frac{S_k-S_{k,1}}{\lam_1}\geq -\frac{S_{k,1}}{\lam_1}\geq \frac{\es}{2k}(-\lam_n)S_k^{11,nn}\geq \frac{\es^2}{2k}\lam_1 S_k^{11,nn} \geq \frac{\es^2 \de }{2k}c(n,k) S_k^{nn}.
    \end{align}

    \textbf{Case 2b:} $\si_{k-1,1}<0$. Then
    \begin{align*}
        S_{k,1}= &\si_{k,1}+\al \si_{k-1,1} \leq \si_{k,1}+\al \frac{k-1}{k}\si_{k-1,1} \quad (\te{use } \si_{k-1,1}<0) \\
        = & \frac{1}{k}\sum_{i>1} \lam_i \si_{k-1,1i}+\frac{\al}{k}\sum_{i>1} \lam_i \si_{k-2,1i}\\
        = & \frac{1}{k}\sum_{i>1} \lam_i \z[\si_{k-1,1}-\lam_i \si_{k-2,1i}\x]+\frac{\al}{k}\sum_{i>1} \lam_i \z[\si_{k-2,1}-\lam_i \si_{k-3,1i}\x] \\
        = & \frac{1}{k} \sum_{i>1} \lam_i \z[\si_{k-1,1}+ \al \si_{k-2,1}\x] -\frac{1}{k} \sum_{i>1} \lam_i^2 \z[\si_{k-2,1i}+\al \si_{k-3,1 i}\x] \\
        = & \frac{1}{k} \sum_{i>1} \lam_i S_k^{11}-\frac{1}{k}\sum_{i>1} \lam_i^2 S_k^{11,ii} \\
        \leq & \frac{n-2}{k}\lam_2 S_k^{11}-\frac{\lam_n^2}{k}S_k^{11,nn} \\
        \leq & \frac{(n-2)\lam_1}{ k \de} \ttt (-\lam_n) S_k^{11,nn}+\frac{\es}{k} \lam_1 \lam_n S_k^{11,nn} \quad (\te{use Case 2})\\
        = & \frac{k+1}{2 k^2}\es \lam_1 \lam_nS_k^{11,nn} \leq  -\frac{\es}{2k}\lam_1 (-\lam_n)S_k^{11,nn}<0.
    \end{align*}
    Thus, inequality \eqref{eq 2.4} also holds, and the proof is complete.
\end{proof}

In the following of this paper, we denote 
$$
S_k^{ij}(D^2 u)=\frac{\pa S_k(D^2 u)}{\pa u_{ij}} \quad \te{and} \quad S_k^{ij,pq}(D^2 u)=\frac{\pa^2 S_k(D^2 u)}{\pa u_{ij}\pa u_{pq}}.
$$

\begin{lemma}\label{lem 2.6}
    Let $u\in C^2(\overline{\ooo})$ be a function satisfying $\lam(D^2 u)\in \tilde{\gaa}_k$. Then 
    \begin{align}\label{eq 2.5}
        -\frac{\al }{n-k+1}\sum_{i=1}^n S_k^{ii}(D^2 u)\leq   \sum_{i=1}^n S_k^{ij}(D^2 u)u_{ij} \leq kS_k(D^2 u).
    \end{align}
\end{lemma}
\begin{proof}
    By Definition \ref{def 2.1} $(5)$ and $\sum_{i,j}S_k^{ij}u_{ij}=k\si_{k}+\al(k-1)\si_{k-1}\leq kS_k$, we obtain
    \begin{equation}\label{eq S_k^ij u_ij+ S_k^ii}
        \sum_{i,j} S_k^{ij} u_{ij}+\frac{\al}{n-k+1}\sum_{i} S_k^{ii}=kS_k+\frac{n-k+2}{n-k+1}\al^2 \si_{k-2}>0.
    \end{equation}
    Thus, \eqref{eq 2.5} holds.
\end{proof}

\begin{lemma}[See \cite{Ball}]\label{lem ball}
    Let $B$ be a symmetric matrix and $A$ be a diagonal matrix with distinct eigenvalues. Denote the eigenvalues of $B$ by $\lam(B)$, and define a $C^2$ function $G$ by $G(B)=g(\lam(B))$, where $g$ is a symmetric function of the eigenvalues $\lam$. Then
    $$
    \z.\frac{dG(A+tB)}{dt}\x|_{t=0}=\sum_{i=1}^n g_i B_{ii},
    $$ 
    and
    \begin{align*}
        \z.\frac{d^2 G(A+tB)}{dt^2 }\x|_{t=0} = \sum_{i,j=1}^n g_{ij} B_{ii}B_{jj} + 2 \sum_{i < j} \frac{g_i -g_j}{\lam_i - \lam_j} B_{ij}^2,
    \end{align*}
    where $g_i=\frac{\pa g(\lam(A))}{\pa \lam_i}$ and $g_{ij}=\frac{\pa^2 g(\lam(A))}{\pa \lam_i \pa \lam_j}$. 
\end{lemma}

\begin{lemma}[G\aa rding-type inequality for $S_k$]\label{lem garding-type ine}
    Let 
    $\lam(D^2 u) \in \tilde{\gaa}_k$ and $\lam(D^2 v)\in \gaa_k$. Then, 
    \begin{equation*}
        \frac{1}{k}S_k^{\frac{1}{k}-1}(D^2 u) \sum_{i,j=1}^n S_k^{ij}(D^2 u) v_{ij}\geq (S_k(D^2 u)+\si_k(D^2 v))^{\frac{1}{k}}-S_k(D^2 u)^{\frac{1}{k}}>0.
    \end{equation*}
    Moreover, we also have
    \begin{equation*}
        S_k(D^2 u)^{-1} \sum_{i,j=1}^n S_k^{ij}(D^2 u) v_{ij} \geq \log (S_k(D^2 u)+\si_k(D^2 v))-\log S_k(D^2 u)>0.
    \end{equation*}
\end{lemma}
\begin{proof}
    For convenience, we denote $A=D^2 u$ and $B=D^2 v$ and claim that 
    \begin{equation}\label{eq S_{k}(A+B)geq S_{k}(A)+sigma_{k}(B)}
        S_k(A+B)\geq S_k(A)+\si_{k}(B).
    \end{equation}

    Using Lemma \ref{lem ball} and the concavity of $(\si_{k-1})^{\frac{1}{k-1}}$ in $\gaa_{k-1}$,
    \begin{equation}\label{eq the concavity of si_{k-1}}
        (\si_{k-1}(A+B))^{\frac{1}{k-1}}\geq (\si_{k-1}(A))^{\frac{1}{k-1}}+(\si_{k-1}(B))^{\frac{1}{k-1}}.
    \end{equation}
    On the other hand, by the concavity of $\frac{\si_{k}}{\si_{k-1}}$ in $\gaa_{k-1}$, we have
    \begin{align}\label{eq the concavity of si_{k}/si_{k-1}}
        \frac{\si_k(A+B)}{\si_{k-1}(A+B)}+\al \geq \frac{\si_k(A)}{\si_{k-1}(A)}+\al+\frac{\si_k(B)}{\si_{k-1}(B)}=\frac{S_k(A)}{\si_{k-1}(A)}+\frac{\si_k(B)}{\si_{k-1}(B)}>0.
    \end{align}
    Thus, by combining \eqref{eq the concavity of si_{k-1}} and \eqref{eq the concavity of si_{k}/si_{k-1}}, we obtain
    \begin{align*}
        S_k(A+B)
        = & \si_{k-1}(A+B)\z[\frac{\si_k(A+B)}{\si_{k-1}(A+B)}+\al \x] \\
        \geq & [\si_{k-1}(A)+\si_{k-1}(B)] \z[\frac{S_k(A)}{\si_{k-1}(A)}+\frac{\si_k(B)}{\si_{k-1}(B)} \x] \\
        >& S_k(A)+\si_{k}(B)>0.
    \end{align*}
    Therefore, \eqref{eq S_{k}(A+B)geq S_{k}(A)+sigma_{k}(B)} holds.
    
    By applying \eqref{eq S_{k}(A+B)geq S_{k}(A)+sigma_{k}(B)} and the concavity of $S_k^{\frac{1}{k}}$ in $\tilde{\gaa}_k$. 
    \begin{align*}
        \frac{1}{k} S_k^{\frac{1}{k}-1} S_k^{ij}(A) B_{ij} \geq&  S_k^{\frac{1}{k}}(A+B)-S_k^{\frac{1}{k}}(A)\geq (S_k(A)+\si_{k}(B))^{\frac{1}{k}}-S_k^{\frac{1}{k}}(A)>0.
    \end{align*}
    Similarly, by the concavity of $\log S_k$ in $\tilde{\gaa}_k$, we have
    \begin{align*}
        S_k^{-1}(A) S_k^{ij}(A) B_{ij} \geq&  \log (S_k(A)+\si_{k}(B)) -\log S_k(A)>0.
    \end{align*}
\end{proof}

The following lemma shows that the condition on $\ooo$ for boundary double normal derivatives estimate (and indeed for the second derivatives estimate) can be relaxed from uniform convexity to uniform $(k-1)$-convexity and almost convexity.

\begin{lemma}\label{lem h=-d+Kd^2}
    Suppose $\ooo \subset \rr^n$ is a $C^2$ uniformly $(k-1)$-convex domain, i.e., $\si_{k-1}(\ka)\geq c_\ka$, and $u\in C^2(\overline{\ooo})$ satisfies $\lam(D^2 u)\in \tilde{\gaa}_k$. Let $d(x)=\operatorname{dist}(x,\pa \ooo)$
    , and define 
    $$
    h(x)=-d(x)+Kd^2(x), \quad \te{in }\ooo_\mu=\{x\in \ooo: d(x)<\mu\}.
    $$
    Then there exist $\mu>0$ small and $K>0$ depending on $n,k,\ooo$, such that 
    \begin{equation*}
        \sum_{i,j=1}^n S_k^{ij}(D^2 u)h_{ij}\geq c_1 \sum_{i=1}^n S_k^{ii}(D^2 u), \quad \te{in } \overline{\ooo_\mu}.
    \end{equation*}
    where the positive constant $c_1$ depends on $n$, $k$, $c_\ka$ and $\ooo$.
\end{lemma}
\begin{proof}
    It follows from \cite[Section 14.6]{GT} that there exists a small constant $0<\mu<\frac{1}{10}$ depending on $\ooo$ such that $d$, $h \in C^2(\overline{\ooo_\mu})$.
    For any $x_0\in \ooo_\mu$, there exists $y_0\in \pa \ooo$ such that $|x_0-y_0|=d(x_0)$. In the principal coordinate system, we have
    \begin{align*}
        -Dd(x_0)&=\nu(y_0)=(0,\cdots,0,1),\\
        -D^2 d(x_0)& = \operatorname{diag}\z\{\frac{\ka_1(y_0)}{1-\ka_1(y_0)d(x_0)}, \cdots,\frac{\ka_{n-1}(y_0)}{1-\ka_{n-1}(y_0)d(x_0)},0\x\}, \\
        D^2 h(x_0)& = \operatorname{diag}\z\{\frac{(1-2Kd(x_0)) \ka_1(y_0)}{1-\ka_1(y_0)d(x_0)}, \cdots,\frac{(1-2Kd(x_0)) \ka_{n-1}(y_0)}{1-\ka_{n-1}(y_0)d(x_0)},2K\x\}.
    \end{align*}

    Since $\si_{k-1}(\ka)\geq c_{\ka}$ on $\pa \ooo$, the Maclaurin inequality yields for any $1\leq i\leq k-2$,
    $$
    \sigma_{i}(\ka)\geq c(n,k) \si_{k-1}(\ka)\geq c(n,k) c_\ka>0, \quad \te{on }\pa \ooo.
    $$
    By choosing $\mu=\mu(n,k,c_\ka,\ooo)>0$ sufficiently small and $K=K(n,k,c_\ka,\ooo)>0$ such that $8K\mu\leq 1$, we obtain in $\overline{\ooo_\mu}$
    $$
    \si_{k-1}(-D^2 d)\geq \frac{3}{4}c_{\ka}, \quad \te{and} \quad \si_{k-1}(D^2 h|nn)\geq (1-2Kd)^{k-1}\si_{k-1}(-D^2 d)\geq \frac{9}{16}c_\ka.
    $$
    Let $D^2 (h-c_1 |x|^2)=\operatorname{diag}\{\zeta_1,\cdots,\zeta_n\}$ and $\zeta'=(\zeta_1,\cdots,\zeta_{n-1})$. Then, in $\overline{\ooo_\mu}$, we have 
    \begin{align*}
        \si_{k}(D^2(h-c_1 |x|^2))
        \geq & 2(K-c_1)\si_{k-1}(\zeta')+\si_k(\zeta') \\
        \geq& c_\ka (K-c_1)-C(n,k,c_1)\geq \frac{c_\ka}{2}K,
    \end{align*}
    where $c_1>0$ is a sufficiently small constant depending on $n$, $k$, $\mu$, $c_\ka$, $K$ and $\ooo$. A similar argument shows that $h-c_1|x|^2\in \gaa_k$ for any $x\in \overline{\ooo_\mu}$.

    Applying Lemma \ref{lem garding-type ine} to $v=h-c_1 |x|^2$, we obtain
    \begin{align*}
        S_k^{-1}(D^2 u) S_k^{ij}(D^2 u) (h-c_1 |x|^2)_{ij} >0, \quad \te{in } \overline{\ooo_\mu},
    \end{align*}
    which implies 
    $$
    \sum_i S_k^{ij}(D^2 u) h_{ij}>c_1 \sum_i S_k^{ii}(D^2 u)  \quad \te{in } \overline{\ooo_\mu}.
    $$
\end{proof}

\section{$C^0$ and $C^1$ Estimates}\label{sec 3-C^0 and C^1}
\setcounter{equation}{0}

In this section we derive $C^0$ estimate and gradient estimate for the $(k-1)$-admissible solution to the Neumann problem \eqref{eq-Neumann for elliptic sum (thm 1.2)}.

\subsection{$C^0$ estimate}

Firstly, following the idea of \cite{Tru-1987}, we can easily obtain the following $C^0$ estimate.
\begin{theorem}\label{thm elliptic C^0}
    Let $\ooo \subset \rr^n$ be a $C^1$ domain, $f\in C(\overline{\ooo})$ satisfy $f>0$, and $\pp\in C(\pa \ooo)$ satisfy $\pp_u\leq c_\pp<0$. If $u\in C^2(\ooo) \cap C^1(\overline{\ooo}) $ is a $(k-1)$-admissible solution to \eqref{eq-Neumann for elliptic sum (thm 1.2)}, then
    $$
        \sup_{\ooo}|u|\leq M_0,
    $$
    where the positive constant $M_0$ depends on $n$, $k$, $c_{\pp}$, $\operatorname{diam}(\ooo)$, $|f|_{C^0}$, and $|\pp(\cdot,0)|_{C^0}$.
\end{theorem}
\begin{proof}
    Without loss of generality, we assume $\sup_{\ooo} u>0$ and $\inf_{\ooo} u<0$.
    Since $\dd u>0$, the (positive) maximum of $u$ is attained at $x_0 \in \pa \ooo$. At $x_0$, we have 
    $$
    0 \leq u_\nu (x_0) =\pp (x_0,u(x_0)) \leq \pp(x_0,0)+c_{\pp} u(x_0),
    $$ 
    which establishes the upper bound for $u$.

    To prove the lower bound for $u$, we define $v=A|x|^2\in \tilde{\gaa}_k$, where $A$ is a sufficiently large constant depending on $n$, $k$, and $|f|_{C^0}$. It follows that
    \begin{align*}
        a^{ij}(u-v)_{ij}=S_k(D^2 u)-S_k(D^2 v)\leq f-(2A)^k C_n^k<0,\quad \te{in }\ooo,
    \end{align*}
    where $(a^{ij})$ is positive definite. Consequently, $u - v$ attains its minimum at $x_1 \in \pa \ooo$. We may assume $u(x_1)<0$, otherwise, the lower bound for $u$ is easily obtained. Hence
    \begin{align*} 
        0 \geq (u -v)_{\nu}(x_1) =\pp(x_1,u(x_1)) -2A x_1 \cdot \nu 
        \geq \pp(x_1,0)+c_{\pp}u(x_1) - 2A \operatorname{diam}(\ooo).
    \end{align*}
    This provides the lower bound for $u$.
\end{proof}

\subsection{Gradient estimate}

In what follows, we establish the gradient estimate for sum Hessian equations with Neumann problems. First, we derive an interior gradient estimate.
Next, we obtain the near boundary gradient estimate in $\ooo_\mu$. 
Before proceeding, we define $M_0=\sup_{\overline{\ooo}}|u|$. 

Following the idea of Chou and Wang \cite{Chou-Wang-CPAM-2001}, we obtain the interior gradient estimate. 
\begin{theorem}\label{thm elliptic interior gradient}
    Let $\ooo$ be a domain in $\rr^n$, and let $f\in C^1(\overline{\ooo}\times [-M_0,M_0])$ satisfy $f>0$. Suppose $u\in C^3(\ooo)$ is a $(k-1)$-admissible solution to sum Hessian equation
    \begin{equation*}
        \si_k (D^2 u )+\al \si_{k-1}(D^2 u )=f(x,u), \quad \te{in } \ooo,
    \end{equation*}
    Then, the following interior estimate holds:
    \begin{equation}\label{eq 3.1}
        \sup_{\ooo \xg \ooo_\mu}|\na u|\leq M_1,
    \end{equation}
    where the positive constant $M_1$ depends on $n$, $k$, $\mu$, $M_0$, $\inf f$, and $|Df|_{C^0}$.
\end{theorem}
\begin{proof}
    For any $z_0 \in \ooo\xg \ooo_{\mu}$, we define the auxiliary function
    \begin{equation*}
        G(x)=|Du| \p(u) \rho(x), \quad \te{in } B_\mu(z_0),
    \end{equation*}
    where $\p(u)=(3M_0-u)^{-\frac{1}{2}}$ and $\rho(x)=\mu^2-|x-z_0|^2$. Suppose $G$ attains its maximum at $x_0\in B_\mu(z_0)$. By rotating the coordinates, we may assume $u_1(x_0)=|Du(x_0)|>0$ and $\{u_{ij}(x_0)\}_{2\leq i,j\leq n}$ is diagonal. 
    Consequently, the function
    $$
    \PPP(x)=\log u_1+\log\p(u)+\log \rho(x)
    $$
    attains local maximum at $x_0\in B_\mu(z_0)$. We further denote
    $$
    \widetilde{\lam}=(\widetilde\lam_2,\cdots,\widetilde\lam_n)=(u_{22}(x_0), \cdots,u_{nn}(x_0)).
    $$
    
    All subsequent calculations are evaluated at $x_0$. At $x_0$,
    \begin{align}
        0= \PPP_i = & \frac{u_{1i}}{u_1}+\frac{\p'}{\p}u_i +\frac{\rho_i}{\rho}, \label{eq 3.2}\\
        0\geq  S_k^{ij} \PPP_{ij}=&  \frac{S_k^{ij} u_{1ij}}{u_1}-S_k^{ij} \frac{u_{1i}u_{1j}}{u_1^2}+\z[\frac{\p''}{\p} -\frac{(\p')^2}{\p^2}\x]S_k^{11} u_1^2  \label{eq 3.3} \\
        & +\frac{\p'}{\p}S_k^{ij} u_{ij} +\frac{S_k^{ij}\rho_{ij}}{\rho}-S_k^{ij} \frac{\rho_i \rho_j}{\rho^2}. \nonumber
    \end{align}
    First, we have $S_{k}^{ij}u_{ij1}=f_{x_1}+f_u u_1$. From \eqref{eq 3.2}, it follows that 
    \begin{align*}
        -S_k^{ij}\frac{u_{1i}u_{1j}}{u_1^2}
        \geq - \frac{(\p')^2}{\p^2} S_k^{11}u_1^2 -2 \sum_{i} S_k^{ii}\frac{\p'u_1}{\p}\frac{|D\rho|}{\rho}-\sum_{i} S_k^{ij}\frac{\rho_i\rho_j}{\rho^2}.
    \end{align*}
    
    Using the inequality $0<\frac{\p'}{\p}\leq \frac{1}{4M_0}$ and \eqref{eq S_k^ij u_ij+ S_k^ii}, we obtain
    \begin{align*}
        \frac{\p'}{\p}\sum_{i,j} S_k^{ij}u_{ij}\geq -\frac{c_{n,k}}{M_0}\sum_{i} S_k^{ii}.
    \end{align*}
    Therefore, \eqref{eq 3.3} becomes 
    \begin{align}\label{eq 3.4}
        0\geq -\frac{|f_{x_1}|}{u_1}-|f_u|+S_k^{11}\frac{u_1^2}{64 M_0^2}-\z[\frac{u_1}{2M_0}\frac{\mu}{\rho}+\frac{c_{n,k}}{M_0} +\frac{2}{\rho}+8\frac{\mu^2 }{\rho^2 } \x]\sum_{i}S_k^{ii}.
    \end{align}

    Without loss of generality, we may assume that 
    \begin{equation}\label{eq 3.5}
        |Du(z_0)|\geq 24\sqrt{2}\frac{M_0}{\mu}.
    \end{equation}
    Then, by $G(x_0)\geq G(z_0)$, we have 
    $$
    u_1(x_0)\rho(x_0)=\frac{G(x_0)}{\p(u(x_0))}\geq\frac{G(z_0)}{\p(u(x_0))}\geq 24M_0 \mu\geq  2\frac{\p}{\p'}|D\rho|(x_0).
    $$
    It follows form \eqref{eq 3.2} that 
    $$
    \frac{u_{11}}{u_1}(x_0)=-\z(\frac{\p'}{\p}u_1+\frac{\rho_1}{\rho}\x)\leq -\frac{\p'}{2\p}u_1<0.
    $$

    Hence, by Lemma \ref{lem 2.4} and \cite[(3.10)]{Chou-Wang-CPAM-2001},  
    \begin{equation}\label{eq 3.6}
        S_k^{11}\geq \frac{1}{n-k+2}\sum_{i} S_k^{ii}\geq c_0(n,k,\inf f)>0.
    \end{equation}
    
    In the end, by combining \eqref{eq 3.4} with \eqref{eq 3.6}, we have
    \begin{align*}
        0\geq -\frac{|f_{x_1}|}{u_1}-|f_u|+\z[\frac{c_{n,k}}{M_0^2}u_1^2-\frac{\mu}{2M_0}\frac{u_1}{\rho}-\frac{c_{n,k}}{M_0} -\frac{2}{\rho}-8\frac{\mu^2 }{\rho^2 } \x]c_0.
    \end{align*}
    This implies that $\rho(x_0)u_{1}(x_0)\leq \widetilde{C}(n,k,\mu,M_0,\inf f ,|Df|_{C^0})$. Consequently, we have
    $$
    |Du(z_0)|\leq \frac{\p(u(x_0))}{\p(u(z_0))}\frac{1}{\rho(z_0)}\rho(x_0)u_1(x_0)\leq \frac{2\widetilde{C}}{\mu^2}:=M_1.
    $$
    Since $z_0\in \ooo\xg \ooo_{\mu}$ is arbitrary and $M_1$ is independent of $z_0$,  the proof is complete.
\end{proof}

Following the idea of Wang \cite{Wang Peihe-2022}, we can obtain the near boundary gradient estimate for the general setting, i.e., the oblique boundary value problem. 

\begin{theorem}\label{thm elliptic near boundary gradient}
    Let $\ooo\subset \rr^n $ be a $C^3$ domain, $f\in C^1(\overline{\ooo}\times [-M_0,M_0])$ satisfy $f>0$, and $\pp\in C^3(\pa \ooo\times [-M_0,M_0])$. If $u \in C^3(\ooo)\cap C^2(\overline{\ooo})$ is the $(k-1)$-admissible solution to  
    \begin{equation*}
        \z\{\begin{aligned}
            & \si_k(D^2 u)+\al \si_{k-1}(D^2 u)=f(x,u)&&  \te{in }\ooo, \\
            & u_\bb=\pp(x,u)&&  \te{on }\pa \ooo,
        \end{aligned}\x.
    \end{equation*}
    where $\bb\in C^{3}(\pa \ooo)$ is a unit vector field of $\pa \ooo$ with $\z<\bb,\nu\x>\geq \de_0$ for some positive constant $\de_0$. Here, $\nu$ denotes the outer unit normal vector field of $\pa \ooo$. Then there exists a small positive constant $\mu$, depending on $n,k,\de_0,\ooo,M_0,\inf f,|f|_{C^1},|\pp|_{C^3}$, and $|\bb|_{C^3}$, such that
    \begin{align}\label{eq 3.7}
        \sup_{\ooo_{\mu}} |D u| \leq \max\{M_1,M_2\},
    \end{align}
    where $M_1$ is defined in Theorem $\ref{thm elliptic interior gradient}$, and the positive constant $M_2$ depends on $n$, $k$, $\mu$, $\de_0$, $\ooo$, $M_0$, $\inf f$, $|f|_{C^1}$, $|\pp|_{C^3}$, and $|\bb|_{C^3}$.
\end{theorem}
\begin{proof}
    We extend the unit vector field $\bb$ smoothly to $\ooo_\mu$ such that $\z<\bb,-Dd\x>=\cos \ttt\geq \de_0$ holds. Define
    \begin{equation*}
        w=u+\frac{\pp d}{\cos \ttt}, \te{ and } \p=|Dw|^2-(Dw \cdot Dd)^2=\sum_{i,j} (\de_{ij}-d_i d_j) w_i w_j\triangleq \sum_{i,j} C^{ij} w_i w_j.
    \end{equation*}
    We introduce the auxiliary function
    \begin{equation*}
        \varPhi =\log \p+g(u)+\al_0 d, \quad \te{in } \ooo_\mu,
    \end{equation*}
    where $g(u)=-\frac{1}{2} \log(1+M_0-u)$, and $\al_0>0$ is a constant to be determined.

    Note that the matrix $\{C^{ij}\}$ is positive semi-definite but not positive definite. Therefore, we need to show that the boundedness of $\p$ implies the boundedness of $|Du|$.
    
    First, using the auxiliary function $\chi=\log |Du|^2+g(u)$, we can establish the estimate
    \begin{equation}\label{eq 3.8}
        \sup_{\ooo} |Du|\leq C_1 (1+\sup_{\pa \ooo}|Du|),
    \end{equation} 
    where $C_1$ depends on $n$, $k$, $M_0$, $\inf f$, and $|Df|_{C^0}$. The proof of \eqref{eq 3.8} is standard and omitted.

    From $(27)-(32)$ of \cite{Wang Peihe-2022}, we have
    \begin{equation}\label{eq 3.9}
        \z|D_\nu w\x|^2 \leq |Dw|^2   \sin^2 \ttt, \quad \te{and} \quad \p\geq |Dw|^2  \cos^2 \ttt\geq \de_0^2 |Dw|^2  \quad \te{on } \pa \ooo.
    \end{equation}
    This implies that for any $x\in \pa \ooo$
    \begin{align}\label{eq 3.10}
        \p(x) \geq \de_0^2 \z|D u(x)-\frac{\pp\nu}{\cos \ttt}\x|^2\geq \frac{\de_0^2}{2}|Du(x)|^2-|\pp|^2_{C^0}.
    \end{align}
    Here, we use the inequality $(x-y)^2\geq \frac{x^2}{2}-y^2$. Combining \eqref{eq 3.8} and \eqref{eq 3.10} yields
    \begin{equation*}
        \sup_{\ooo} |Du|^2\leq C(n,k,\de_0,M_0,\inf f, |Df|_{C^0}, |\pp|_{C^0})(1+\sup_{\pa \ooo} \p).
    \end{equation*}

    From the above argument, it suffices to prove the bound for $\p$. Suppose the maximum of $\PP$ over $\overline{\ooo_\mu}$ is attained at $x_0$. By the interior gradient estimates (Theorem \ref{thm elliptic interior gradient}), it suffices to consider the following two cases.

    \textbf{Case 1:} $x_0 \in \pa \ooo$.
    The proof of this case relies solely on the Neumann boundary condition, and $\al_0=\al_0(\pa \ooo,\de_0,|\bb|_{C^2},|\pp|_{C^0})>0$ has been chosen appropriately. Thus, we omit the proof, see \cite[Theorem 4.1]{Wang Peihe-2022} for details.

    \textbf{Case 2:} $x_0 \in \ooo_\mu$.
    By \eqref{eq 3.8}, we may assume that $|Du|\big|_{\pa \ooo}$ attains its maximum at $x_1$. Furthermore, we assume that both $|Du(x_0)|$ and $|Du(x_1)|$ are sufficiently large, ensuring that $|D w(x_0)|$ is comparable to $|Du(x_0)|$, and $|D w(x_1)|$ is comparable to $|Du(x_1)|$. 
    Otherwise, the proof is complete.

    From \eqref{eq 3.8}, \eqref{eq 3.9} and $\PP(x_0)\geq \PP(x_1)$, we obtain 
    \begin{align}
        \p(x_0)\geq &  C_2 \p(x_1)\geq C_2 \de_0^2 |Dw(x_1)|^2\geq  \frac{1}{2}C_2 \de_0^2 |Du(x_1)|^2 \nonumber \\
        \geq & \frac{1}{2}C_2 \de_0^2 \z(\frac{1}{C_1}\sup_{\ooo}|Du|-1\x)^2 
        \geq   \frac{C_2\de_0^2}{4C_1^2}|Du(x_0)|^2 
        \geq  
        C_3 |Dw(x_0)|^2,  \label{eq 3.11}
    \end{align}
    where $C_2>0$ depends on $\al_0$, $\mu$, and $M_0$. 

    Now choose coordinates such that $D^2 u(x_0)$ is diagonal, all subsequent calculations are evaluated at $x_0$.
    
    For $1\leq j \leq n$, define ${T_j} = \sum_{i} {{C^{ij}}{w_i}}$ and let $\mathcal{T}=(T_1,\,T_2,\cdots,T_n)$. It is straightforward to verify that $|\mathcal{T}| \leq |Dw|$ and $ \phi=\langle\mathcal{T},\ Dw \rangle$. Moreover, from \eqref{eq 3.11}, we have 
    \begin{equation}\label{eq 3.12}
        C_{3}|Dw| \leq |\mathcal{T}| \leq |Dw|.
    \end{equation}

    Without loss of generality, we may assume that
    \begin{equation}\label{eq 3.13}
        T_1 w_1 \geq \frac{C_3 }{n}|D w|^2,\quad \te{and} \quad T_1>0,
    \end{equation}
    which implies
    \begin{equation*} 
        \frac{w_1}{T_1} \geq \frac{C_3}{n}.
    \end{equation*}
    It follows that $w_1$, $u_1$, $T_1$, $|Dw|$, $|Du|$, $|\mathcal{T}|$, and $\p^{\frac{1}{2}}$ are mutually comparable at $x_0$. Therefore,
    \begin{equation}\label{eq 3.14}
        \frac{{{u_1}}}{{{T_1}}} \geq \frac{C_3}{2n}.
    \end{equation}

    By direct calculation and the Cauchy inequality, 
    \begin{align}
        0=\PP_i =& \frac{\phi_i}{\phi} + g' u_i+\al_0 d_i,\label{eq thm 3.3 PP_i=0} \\
        0\geq S_k^{ii}\PP_{ii} =&  S_k^{ii}\frac{\p_{ii}}{\p}- S_k^{ii}(g'{u_i} + \al_0 {d_i})^2 +g' S_k^{ii}u_{ii} + g'' S_k^{ii}u_i^2 + \al_0 S_k^{ii}d_{ii}\nonumber \\
        \geq  & S_k^{ii}\frac{\p_{ii}}{\p} +\z[g''-\frac{3}{2}(g')^2\x] S_k^{ii} u_i^2+g' S_k^{ii} u_{ii}-2\al_0^2 S_k^{ii} d_i^2+\al_0 S_k^{ii} d_{ii}.\nonumber\\
        =& \operatorname{I} + \operatorname{II} + \operatorname{III} + \operatorname{IV}+\operatorname{V}. \label{eq thm 3.3 I+II+III+IV+v}
    \end{align}

    We first show that $u_{11}<0$.
    From \eqref{eq thm 3.3 PP_i=0} and the definitions of $\p$ and $T_l$, 
    \begin{equation}\label{eq 3.16}
        \sum_{l} T_l w_{l1}  = -\frac{\phi }{2}(g'{u_1} + \al_0 {d_1}) - \frac{1}{2} \sum_{i,j} {C^{ij}}_{,1}{w_i}{w_j},
    \end{equation}
    By direct calculation, we have
    \begin{align}\label{eq 3.17}
            w_i =& {u_i}\z(1 + \frac{\pp_u d}{\cos \ttt}\x) +  \frac{d}{\cos\theta} \pp_{x_i}+ \pp\z(\frac{d}{\cos\theta}\x)_i, \nonumber\\
            w_{ij} =& u_{ij}\z(1 + \frac{\pp_u d}{\cos \ttt}\x) + \frac{{{\pp_{uu}}d}}{{\cos \theta }}{u_i}{u_j} + \frac{\pp_{u x_j} d}{\cos\theta} u_i + \pp_u u_i \z(\frac{d}{{\cos \theta }}\x)_j + {\pp_u}{u_j}\z(\frac{d}{{\cos \theta }}\x)_i\\
            &+ \frac{ \pp_{x_i x_j} d }{\cos \theta } + \frac{\pp_{u x_i} d}{\cos\theta} u_j+ \z(\frac{d}{{\cos \theta }}\x)_i \pp_{x_j} + \z(\frac{d}{{\cos \theta }}\x)_j \pp_{x_i} + \pp \z(\frac{d}{{\cos \theta }}\x)_{ij}. \nonumber
    \end{align}
    Hence, by \eqref{eq 3.12}, \eqref{eq 3.14}, and \eqref{eq 3.17}, equality \eqref{eq 3.16} becomes
    \begin{align*}
        {u_{11}}\z(1+ \frac{\pp_u d}{\cos \theta }\x) \leq - \frac{{{u_1}}}{{2{T_1}}}g'\phi + C_4 (d|Dw|^2 + |Dw|) <0,
    \end{align*}
    provided $\mu>0$ is a sufficiently small constant. Thus, by Lemma \ref{lem 2.4}, 
    \begin{equation}\label{eq thm 3.3 S_k^11 geq sum S_k^ii}
        S_k^{11} \geq c_{n,k} \sum_{i} S_k^{ii}.
    \end{equation}

    We now deal with the last four terms. Using \eqref{eq thm 3.3 S_k^11 geq sum S_k^ii} and Lemma \ref{lem 2.6}, 
    \begin{equation}\label{eq 3.20}
        \begin{split}
            & \operatorname{II}\geq  \z[g''-\frac{3}{2}(g')^2\x]S_k^{11}u_1^2 \geq \frac{c_{n,k}|Dw|^2 }{8(1+2M_0)^2} \sum_{i=1}^n S_k^{ii},  \\
            & \operatorname{III}+\operatorname{IV}+ \operatorname{V} \geq -C (1+\al_0+\al_0^2) \sum_{i=1}^n S_k^{ii}:=-C_5 \sum_{i=1}^n S_k^{ii}.
        \end{split}
    \end{equation}
    
    For the term $\operatorname{I}$,
    \begin{align*}
        \p \operatorname{I}=&  \sum\limits_{i,p,q } S_k^{ii} {C^{pq}}_{,ii} w_p w_q+2\sum_{i,p,q } C^{pq} S_k^{ii} w_{iip} w_q \\
        & +4 \sum_{i,p,q} {S_k^{ii} {C^{pq}}_{,i} w_{ip} w_q}+2\sum_{i,p,q} S_k^{ii} C^{pq} w_{ip} w_{iq}\\
        =& \p[\operatorname{I}_1 + \operatorname{I}
        _2+\operatorname{I}_3+\operatorname{I}_4].
    \end{align*}
    We consider these four terms one by one in the following.

    For the term $\operatorname{I}_1$, it is easy to deduce that
    \begin{equation}\label{eq 3.21}
        \p \operatorname{I}_1 =   \sum_{i,p,q} S_k^{ii} {C^{pq}}_{,ii} w_p w_q \geq -C_6 |Dw|^2 \sum_{i} S_k^{ii}.
    \end{equation}

    For the term $\operatorname{I}_2$, we need more careful computations,
    \begin{align*}
        \phi \operatorname{I}_2=&2 \sum_{i,p,q} C^{pq}w_q S_k^{ii}w_{iip} = 2\sum_{i,p} T_p S_k^{ii} \z(u+\frac{{\pp d}}{{\cos \theta }}\x)_{iip} \\
        = & 2T_p (f_{x_p}+f_u u_p)+2\sum\limits_{i,p} T_p S_k^{ii}\z(\frac{{\pp d}}{{\cos \theta }}\x)_{iip}.
    \end{align*}
    To proceed, we should compute $(\frac{{\pp  d}}{{\cos \theta }})_{ijp}$. By direct calculation,
    \begin{align}\label{eq 3.22}
        \z(\frac{\pp (x,u) d }{\cos \theta }\x)_{ijp}=&  (\pp)_{ijp} \frac{ d}{\cos \theta } + (\pp)_{ij}\z(\frac{d}{\cos \theta }\x)_p + (\pp)_{ip}\z(\frac{d}{\cos \theta }\x)_j + (\pp)_{jp}\z(\frac{d}{\cos \theta}\x)_i \\
        &+ (\pp)_{i}\z(\frac{d}{\cos \theta }\x)_{jp} + (\pp)_{j} \z(\frac{d}{\cos \theta }\x)_{ip} +(\pp)_{p}\z(\frac{d}{\cos \theta}\x)_{ij} + \pp \z(\frac{d}{\cos \theta}\x)_{ijp}, \nonumber
    \end{align}
    where 
    \begin{align}\label{eq 3.23}
        {(\pp )_i} =& {\pp_{x_i}} + {\pp_u}{u_i}, \nonumber \\
        {(\pp )_{ij}} =&  {\pp_{x_ix_j}} + {\pp_{x_iu}}{u_j} + {\pp_{ux_j}}{u_i} + {\pp_{uu}}{u_i}{u_j} + {\pp_u}{u_{ij}}, \nonumber \\
        {(\pp )_{ijl}} =& {\pp_{x_ix_jx_l}} + {\pp_{x_ix_ju}}{u_l} + {\pp_{x_iux_l}}{u_j} + {\pp_{x_iuu}}{u_j}{u_l} + {\pp_{x_iu}}{u_{jl}} \\
        &+ {\pp_{ux_jx_l}}{u_i} + {\pp_{ux_ju}}{u_i}{u_l} + {\pp_{ux_j}}{u_{il}} \nonumber\\
        &+ {\pp_{uux_l}}{u_i}{u_j} + {\pp_{uuu}}{u_i}{u_j}{u_l} + {\pp_{uu}}{u_{il}}{u_j} + {\pp_{uu}}{u_i}{u_{jl}} \nonumber \\
        &+ {\pp_{ux_l}}{u_{ij}} + {\pp_{uu}}{u_l}{u_{ij}} + {\pp_u}{u_{ijl}}.\nonumber
    \end{align}
    Therefore  we have
    \begin{equation}\label{eq 3.24}
        \phi \operatorname{I}_2 \geq  -C_7 (|Dw|^3+d|Dw|^4)\sum_{i} S_k^{ii} -C_7 \z(|Dw|+d|Dw|^2\x)\z(1+\sum_{i} |S_k^{ii}u_{ii}|\x).
    \end{equation}
    Almost the same procedure, we can deal with the remained two terms.
    \begin{align}\label{eq 3.25}
        \phi \operatorname{I}_3 =& 4 \sum_{i,p,q} {C^{pq}}_{,i} w_q S_k^{ii} w_{ip} = 4\sum_{i,p,q} {C^{pq}}_{,i} w_q S_k^{ii}\z(u+\frac{\pp d}{\cos \ttt}\x)_{ip} \nonumber\\
        \geq & -C_8 (|Dw|^2+d|Dw|^3) \sum_{i} S_k^{ii} - C_8 |Dw|\sum\limits_{i} |S_k^{ii} u_{ii} |,
    \end{align}
    and
    \begin{align}\label{eq 3.26}
        \p \operatorname{I}_4 = & 2\sum_{i,p,q}  S_k^{ii} C^{pq} \z(u+\frac{\pp d}{\cos \ttt}\x)_{ip} \z(u+\frac{\pp d}{\cos \ttt}\x)_{iq} \nonumber\\
        =& 2\sum_{i} C^{ii}S_k^{ii}u_{ii}^2+4\sum_{i,p} C^{ip}S_k^{ii}u_{ii}\z(\frac{\pp d}{\cos \ttt}\x)_{ip}+2\sum_{i,p,q} S_k^{ii}C^{pq}\z(\frac{\pp d}{\cos \ttt}\x)_{ip}\z(\frac{\pp d}{\cos \ttt}\x)_{iq} \nonumber \\
        = & \p[\operatorname{I}_{41}+\operatorname{I}_{42}+\operatorname{I}_{43}].
    \end{align}

    Since $\{C^{pq}\}$ is positive semi-definite, $\p \operatorname{I}_{43}\geq 0$. Using \eqref{eq 3.23}, it follows that
    \begin{equation*}
        \p \operatorname{I}_{42}\geq -C_9(|Dw|+d|Dw|^2)\sum_{i} |S_k^{ii}u_{ii}|-4|\pp_u| \frac{d}{\cos \ttt} \sum_{i} C^{ii}S_k^{ii}u_{ii}^2.
    \end{equation*}
    Thus, if $\mu$ is sufficiently small, equation \eqref{eq 3.26} becomes
    \begin{equation}\label{eq 3.27}
        \p \operatorname{I}_4\geq \sum_{i} C^{ii}S_k^{ii}u_{ii}^2-C_9(|Dw|+d|Dw|^2)\sum_{i} |S_k^{ii}u_{ii}|.
    \end{equation}

    Taking into account \eqref{eq 3.21}, \eqref{eq 3.24}, \eqref{eq 3.25} and \eqref{eq 3.27}, and using Lemma \ref{lem 2.3}, we obtain
    \begin{align}\label{eq 3.28}
        \p \operatorname{I}\ge & \sum_{i} C^{ii} S_k^{ii} u_{ii}^2  -C_{10} (|Dw|+d|Dw|^2) \sum_{i} |S_k^{ii}u_{ii}| \\
        &-C_{10} (|Dw|^3+d|Dw|^4)\sum_{i} S_k^{ii}.\nonumber
    \end{align}
    
    In what follows, we further simplify the term $\operatorname{I}$. Let $C^{i_0 i_0}=\min_{i} C^{ii}\geq 0$. Then, for any $i\ne i_0$, we have $C^{ii}\geq \frac{1}{2}$. Otherwise, it follows that $\sum_{i} C^{ii}<\frac{1}{2}\times 2+(n-2)=n-1$, which contradicts the identity $\sum_{i=1}^n C^{ii}=n-1$. Using Lemmas \ref{lem 2.3} and \ref{lem 2.6}, 
    \begin{align*}
        |S_k^{i_0 i_0}u_{i_0 i_0}|
        \leq & kS_k +C_{n,k} \sum_{i} S_k^{ii} +\sum_{i\ne i_0} |S_k^{ii}u_{ii}|\leq C_{11} \sum_{i} S_k^{ii} +\sum_{i\ne i_0} |S_k^{ii}u_{ii}|.
    \end{align*}
    Plugging this into \eqref{eq 3.28} and applying the inequality $ax^2+bx\geq -\frac{b^2}{4a}$, we have
    \begin{align}\label{eq 3.29}
        \p \operatorname{I}
        \geq & \frac{1}{2}\sum_{i\ne i_0}S_k^{ii} \z[u_{ii}^2- 2  C_{12}(|Dw|+d|Dw|^2)|u_{ii}|\x] - C_{12}(|Dw|^3+d|Dw|^4) \sum_{i} S_k^{ii} \nonumber  \\
        \geq & -C_{13} (|Dw|^3+d|Dw|^4) \sum_{i} S_k^{ii}.
    \end{align}

    Combining \eqref{eq 3.20} with \eqref{eq 3.29}, we obtain
    \begin{equation*}
        0\geq \z[\z(\frac{c_{n,k}}{8(1+2M_0)^2}-C_{14}d \x)|Dw|^2-C_{14}|Dw|-C_5 \x] \sum_{i} S_k^{ii}.
    \end{equation*}
    Therefore, the proof is complete provided that $\mu>0$ is chosen sufficiently small.
\end{proof}

\section{Second-Order Derivatives Estimate}\label{sec 4-second derivative}

In this section, we primarily establish the global second derivative estimate for the Neumann problem
\begin{equation}\label{eq 4.1}
    \left\{\begin{aligned}
        & \si_k(D^2 u)+\al \si_{k-1}(D^2 u )=f(x,u) && \te{in } \ooo, \\
        & u_\nu =\pp(x,u) && \te{on } \pa \ooo,
    \end{aligned}\right.
\end{equation}
where $f\in C^2(\overline{\ooo}\times[-M_0,M_0])$ and $\pp\in C^3(\pa \ooo\times[-M_0,M_0])$ satisfy
\begin{equation}\label{eq 4.2}
    f>0, \quad \te{and} \quad \pp_{u}\leq c_{\pp}<0.
\end{equation}

The main theorem of this section is as follows.

\begin{theorem}\label{thm elliptic C^2 estimates}
    Let $\Omega \subset \mathbb{R}^n$ be a $C^4$ almost convex and uniformly $(k-1)$-convex domain, and let $f, \pp$ satisfy \eqref{eq 4.2}, with $a_\ka>\frac{1}{2}c_\pp$. If $u \in C^4(\Omega)\cap C^3(\overline \Omega)$ is the $(k-1)$-admissible solution to \eqref{eq 4.1}. Then 
    \begin{align*}
        \sup_{\Omega} |D^2 u|  \leq M_3,
    \end{align*}
    where the positive constant $M_3$ depends on $n, k, c_\pp, a_\ka, c_\ka, \Omega, |Du|_{C^0}, \inf f, |f|_{C^2}$, and $|\pp|_{C^3}$.

    In particular, when $\pp$ merely satisfies $\pp_u\leq 0$, we must require $\ooo$ to be uniformly convex. 
\end{theorem}

By Lemma \ref{lem 2.5}, we may derive that second-order derivatives estimate which is similar to Ma-Qiu \cite{Ma-Qiu-2019-CPM} and Chen-Zhang \cite{Chen-Zhang-Bull-2021}, see also \cite{LTU}.

\subsection{Reduce global second estimates to boundary double normal derivatives}

\begin{lemma}\label{lem elliptic global C^2 reduce to boundary}
    Let $\Omega \subset \mathbb{R}^n$ be a $C^4$ almost convex domain, and let $f, \pp$ satisfy \eqref{eq 4.2}, with $a_{\ka}>\frac{1}{2}c_{\pp}$. If $u \in C^4(\Omega)\cap C^3(\overline \Omega)$ is the $(k-1)$-admissible solution to \eqref{eq 4.1}. Then 
    \begin{align}\label{eq 4.3}
        \sup_{\Omega} |D^2 u|  \leq M_4(1+ \max_{\pa \ooo} |u_{\nu \nu}|),
    \end{align}
    where the positive constant $M_4$ depends on $n$, $k$, $c_{\pp}$, $a_\ka$, $\Omega$, $|Du|_{C^0}$, $\inf f$, $|f|_{C^2}$, and $|\pp|_{C^3}$.

    In particular, when $\pp$ merely satisfies $\pp_u\leq 0$, we must require $\ooo$ to be uniformly convex. 
\end{lemma}
\begin{proof}
    The proof follows the same approach as in \cite[Lemma 13]{Ma-Qiu-2019-CPM}, and it relies on Lemma \ref{lem 2.2} $(b)$ and $(c)$. We emphasize that equation (38) in \cite{Ma-Qiu-2019-CPM} can be modified as
    $$
    u_{\xi \xi \nu}\leq (\pp_{u}-2D_1 \nu^1)u_{\xi \xi}+C(1+|u_{\nu \nu}|)\leq (c_{\pp}-2a_{\ka})u_{\xi\xi}+C(1+|u_{\nu \nu}|),
    $$
    where the almost convexity condition of $\ooo$ ensures $2a_{\ka}>c_{\pp}$.
\end{proof}

\subsection{An upper estimate of the double normal derivatives on the boundary}
\

Due to the lack of homogeneity, the double normal derivatives estimate is a bit more complicated. So we give the proof here.

\begin{lemma}\label{lem elliptic u_nu nu leq C}
    Under the assumptions of Theorem $\ref{thm elliptic C^2 estimates}$, then we have
    \begin{align}\label{eq elliptic u_{nu nu} leq C}
        \max_{\pa \ooo} u_{\nu \nu}  \leq M_5,
    \end{align}
    where the positive constant $M_5$ depends on $n, k, c_\ka, \Omega, M_4, |Du|_{C^0}, \inf f, |f|_{C^2}$, and $|\pp|_{C^3}$.

    In particular, when $\pp$ merely satisfies $\pp_u\leq 0$, we must require $\ooo$ to be uniformly convex. 
\end{lemma}
\begin{proof}
    Without loss of generality, we may assume $M:=\max\limits_{\pa \ooo}|u_{\nu \nu}|=\max\limits_{\pa \ooo}u_{\nu \nu}>0$. Otherwise, the result follows directly from the argument in Lemma \ref{lem elliptic u_nu nu geq -C}. Furthermore, we assume that $u_{\nu \nu}|_{\pa \ooo}$ attains its maximum at $z_0$. 

    According to Lemma \ref{lem h=-d+Kd^2}, there exist a small constant $\frac{1}{10}>\mu>0$, depending on $n,k,\ooo$, and a constant $K>0$ such that $8K\mu\leq  1$ and $h(x)=-d(x)+Kd^2(x)$ satisfies 
    \begin{align*}
        \sum_{i,j} S_k^{ij}h_{ij} \geq c_1 \sum_{i} S_k^{ii}, \quad \te{in } \ooo_\mu,
    \end{align*}
    where $c_1>0$ depends on $n$, $k$, $c_\ka$, and $\ooo$.

    Let $\bb$ and $A$ be constants to be determined, and define the auxiliary function
    \begin{align*}
        P(x)=(1+\beta d)\z[D u \cdot (-D d)-\pp(x,u) \x] - \z(A + \frac{M}{2}\x)h(x),\quad \te{in } \ooo_\mu.
    \end{align*}

    Clearly, $P = 0$ on $\pa \ooo$. 
    On $\pa \ooo_\mu \cap \ooo$, we have $d=\mu$ and
    \begin{align*}
        P(x) \geq& - C_{1}-\mu \z[\bb C_{1}-\z(A+\frac{M}{2}\x)(1-K \mu)\x] > 0,
    \end{align*}
    provided $A\geq \frac{8}{7}C_{1}\z(\bb+\mu^{-1}\x)$, with $C_1=C_1(|Du|_{C^0},|\pp|_{C^0})$.

    We claim that $P\big|_{\overline{\ooo_{\mu}}}$ attains its minimum only on $\pa \ooo$. Since $P|_{\pa\ooo}=0$, then 
    \begin{align*}
        0 \geq P_\nu(z_0)=&\z[u_{\nu \nu}(z_0) -\sum_j u_j \nu_i d_{ji}-D_\nu \pp \x] -\z(A + \frac{M}{2}\x) \notag\\
        \geq& \max_{\pa \ooo} u_{\nu \nu}-C_{2}- A  - \frac{M}{2}=\frac{1}{2}\max_{\pa \ooo}u_{\nu\nu}-C_{2}-A,
    \end{align*}
    where $C_{2}>0$ depends on $\ooo$, $|Du|_{C^0}$, and $|\pp|_{C^1}$.
    Therefore, \eqref{eq elliptic u_{nu nu} leq C} holds.

    To prove this claim, we assume by contradiction that $P$ attains its minimum at some point  $x_0 \in \Omega_\mu$. Moreover, assume that $D^2 u(x_0)$ is diagonal. All subsequent calculations are evaluated at $x_0$. Through direct calculation, we obtain
    \begin{align}\label{eq elliptic upper bound 0=P_i}
        0 = P_i =& \bb d_i \z[-\sum_{j} u_j d_j  -\pp\x] + (1+ \bb d)\z[-\sum_{j} (u_{ji} d_j + u_j d_{ji}) -\pp_{x_i}-\pp_u u_i \x]  \\
        &- \z(A + \frac{M}{2}\x)h_i, \notag
    \end{align}
    and
    \begin{align}
        0 \leq {P}_{ii}  
        =& -\beta d_{ii} \z[\sum_{j} u_j d_j + \pp\x] - 2\beta d_i \z[u_{ii} d_i+\sum_{j}u_j d_{ji}+\pp_{x_i}+\pp_u u_i\x]-\z(A+\frac{M}{2}\x)h_{ii} \notag \\
        &- (1+ \beta d)\z[\sum_{j} u_{jii} d_j+2 u_{ii} d_{ii}+\sum_{j} u_j d_{jii}+\pp_{x_i x_i}+2\pp_{x_i u}u_i+\pp_{uu}u_i^2+\pp_u u_{ii}\x] \notag \\
        \leq&  -2\beta u_{ii} d_i^2 -(1+ \beta d)\z[\sum_{j} {u_{jii} d_j }+ 2u_{ii} d_{ii} +\pp_u u_{ii}\x]- \z(A +\frac{M}{2}\x)h_{ii} \label{eq elliptic upper bound 0 leq P_{ii}} \\
        & + C_{3}(\ooo,|Du|_{C^0},|\pp|_{C^2})(1+\bb).\notag
    \end{align}
    Contacting $S_k^{ii}$ in both sides of \eqref{eq elliptic upper bound 0 leq P_{ii}} and using $\sum_{i} S_k^{ii}\geq c_0(n,k,\inf f)$, it follows that
    \begin{align}\label{eq elliptic upper bound,  0 leq S_k^ii P_ii}
        0 \leq \sum_{i} S_k^{ii} P_{ii}  
        \leq&  - 2\beta \sum_{i} S_k^{ii} u_{ii} d_i^2- 2(1+ \beta d)\sum_{i} S_k^{ii} u_{ii} d_{ii}-(1+\bb d )\pp_u \sum_{i} S_k^{ii}u_{ii}   \\
        &+ \z[-\z(A + \frac{M}{2}\x)c_1  + C_{4}(1+\bb)\x]\sum_{i} S_k^{ii}\notag
    \end{align}
    where the constant $C_{4}$ depends on $n, k, \ooo, |Du|_{C^0}, \inf f, |f|_{C^1}$, and $|\pp|_{C^2}$.

    We now partition the indices $1 \leq  i \leq  n$ into two categories to analyze the term $\sum_i S_k^{ii}u_{ii}d_i^2$. Let
    $$
    B = \z\{ i:\beta d_i ^2  < \frac{1}{n},1 \leq i \leq n\x\} \quad \te{and} \quad  G= \z\{i:\beta d_i ^2 \geq  \frac{1}{n},1 \leq i \leq n\x\}.
    $$ 
    By choosing $\beta = \frac{1}{\mu} >1$, $G$ is nonempty. Furthermore, there exists an index $i_0 \in G$, and without loss of generality, we may assume $i_0=1$, such that 
    $$
    d_1^2 \geq \frac{1}{n}.
    $$
    
    Using \eqref{eq elliptic upper bound 0=P_i}, we have
    \begin{align*}
        u_{ii}= \frac{1-2Kd}{1+ \beta d}\z(A + \frac{M}{2}\x) + \frac{\beta \z[-\sum_{j} u_j d_j -\pp \x]}{1+ \bb d} + \frac{-\sum_{j} u_j d_{ji} - D_i \pp}{d_{i } }.
    \end{align*}
    Choosing a sufficiently large $A >0$, such that for any $i \in G$
    \begin{align*} 
        \z| \frac{\beta \z[-\sum_j u_j d_j -\pp \x]}{1+ \beta d} +\frac{-\sum_{j} u_j d_{ji} -D_i\pp}{d_{i}}\x| 
        \leq C_{5}\z(\bb +\sqrt{n\bb}\x)
        \leq \frac{A}{8},
    \end{align*}
    where $C_{5}$ depends on $\ooo$, $|Du|_{C^0}$, and $|\pp|_{C^1}$. Hence
    \begin{align}\label{eq elliptic upper bound u_{ii} >>1}
        \frac{A+M}{8} \leq u_{ii} \leq \frac{9A}{8}+ \frac{M}{2}, \quad \te{for any } i \in G.
    \end{align}
    
    Thus, \eqref{eq elliptic upper bound,  0 leq S_k^ii P_ii} becomes
    \begin{align}
        0 \leq \sum_{i} {S_k^{ii} P_{ii} }  \leq&  - 2\beta \sum\limits_{i \in G} {S_k^{ii} u_{ii} d_i ^2 }  - 2\beta \sum\limits_{i \in B} {S_k^{ii} u_{ii} d_i ^2 }  \notag \\
        &- 2(1+ \beta d)\sum\limits_{u_{ii}  > 0} {S_k^{ii} u_{ii} d_{ii} } -\pp_u(1+\bb d)\sum_{u_{ii}>0} S_k^{ii}u_{ii} \notag \\
        &+ \z[-\z(A + \frac{M}{2}\x)c_1  + C_{4}\bb\x]\sum\limits_{i } {S_k^{ii} }   \notag \\
        \leq & -\frac{2\bb}{n}S_k^{11}u_{11}-\frac{2}{n}\sum_{i\in B;u_{ii}<0}S_k^{ii}u_{ii}+(4\ka_{\operatorname{max}}+2|\pp_u|)\sum_{u_{ii}>0}S_k^{ii}u_{ii} \notag \\
        & + \z[-\z(A+\frac{M}{2}\x)c_1+C_{4}\bb\x]\sum_{i} S_k^{ii}.\notag \\
        \triangleq & \operatorname{I}+\operatorname{II}+\operatorname{III}+\operatorname{IV}. \label{eq elliptic upper bound, I+II+III+IV}
    \end{align}

    In the following, we divide into three cases. Without loss of generality, we assume that $u_{22} \geq \cdots \geq u_{nn}$.

    \textbf{Case 1:} $u_{nn} \geq 0$.

    In this case, we observe that $\operatorname{I}<0$, $\operatorname{II}=0$, and the term $\operatorname{III}$ satisfies
    \begin{align*}
        \operatorname{III}=(4\kappa_{\operatorname{max}}+ 2|\pp_u|) \sum_{i} {S_k^{ii} u_{ii} } \leq  (4\ka_{\operatorname{max}}+2|\pp_u|)kS_k \leq C_{6}(k,\ooo,|f|_{C^0},|D\pp|_{C^0}).
    \end{align*}
    We use the fact that $\sum_{i=1}^n S_k^{ii}\geq  c_0$ and choose a sufficiently large constant $A$ satisfying
    $$
    A>\frac{c_0C_{4}\bb+C_{6}}{c_0c_1},
    $$ 
    then \eqref{eq elliptic upper bound, I+II+III+IV} becomes
    \begin{align*} 
        0 \leq \sum_{i} S_k^{ii} P_{ii} \leq C_{6}+ \z[- \z(A + \frac{M}{2}\x)c_1 + C_{4}\bb\x]c_0 <0.
    \end{align*}
    This leads to a contradiction.

    \textbf{Case 2:} $u_{nn} < 0$ and $-u_{nn} < \frac{c_1}{10 \z(4\kappa_{\operatorname{max}}+2|D\pp|_{C^0}+ \frac{2}{n}\x)}u_{11}$.

    In this case, $\operatorname{I}<0$. Form \eqref{eq elliptic upper bound u_{ii} >>1}, we obtain
    \begin{align}
        \operatorname{II}+ \operatorname{III}
        =&  (4\kappa_{\operatorname{max}}+ 2|\pp_u|) \z(\sum_{i} {S_k^{ii} u_{ii} }-\sum_{u_{ii}<0}S_k^{ii}u_{ii}\x)-\frac{2}{n}\sum_{i\in B;u_{ii}<0}S_k^{ii}u_{ii} \notag \\
        \leq &  (4\kappa_{\operatorname{max}}+ 2|\pp_u|) kS_k -\z(4\kappa_{\operatorname{max}}+ 2|\pp_u|+\frac{2}{n}\x) \sum_{u_{ii}<0}S_k^{ii}u_{ii}. \label{eq elliptic upper bound, case 2-II+III-middle}
    \end{align}
    Therefore \eqref{eq elliptic upper bound, I+II+III+IV} becomes
    \begin{align*}
        0 \leq \sum_{i} S_k^{ii} P_{ii} 
        \leq&  C_{6}+\frac{c_1}{10}u_{11}\sum_{u_{ii}<0} S_k^{ii} \notag  + \z[- \z(A + \frac{M}{2}\x)c_1 + C_{4}\bb\x]\sum_{i}  {S_k^{ii} }  \notag \\
        \leq &  C_{6}+\z[-c_1\z(\frac{71}{80}A+\frac{9}{20}M\x)+C_{4}\bb\x]c_0<0,
    \end{align*}
    provided $A>\frac{80\z[C_{6}+ c_0 C_{4}\bb\x]}{71 c_0 c_1}$. This is a contradiction.

    \textbf{Case 3:} $u_{nn} < 0$ and $-u_{nn} \geq \frac{c_1}{10 \z(4\kappa_{\operatorname{max}}+2|D \pp|_{C^0}+ \frac{2}{n}\x)} u_{11} $.

    In the following, we divide into three subcases to deal with.
    
    \textbf{Case 3a:} $k=2$ and $0\geq u_{22}(\geq u_{33} \geq  \cdots \geq u_{nn}$).

    Since $0<S_k^{11}=S_2^{11}=\si_{1,1}+\al=\sum_{i\geq2} u_{ii}+\al$, then for any $2\leq i\leq n$, we have
    $$
    -\al \leq u_{ii}\leq 0. 
    $$
    If follows from \eqref{eq elliptic upper bound, case 2-II+III-middle} that
    \begin{align*}
        \operatorname{II} +\operatorname{III}\leq & C_{6}+\z(4\ka_{\operatorname{max}}+2|\pp_u|+\frac{2}{n}\x) \al \sum_{i} S_k^{ii} \leq C_{7}(n,\ooo,|f|_{C^0},|D\pp|_{C^0}) \sum_{i} S_k^{ii}.
    \end{align*}
    Hence, if we choose $A>C_{4}\bb +C_{7}$,
    \begin{align*}
        0\leq \sum_i S_k^{ii}P_{ii}\leq \z[-\z(A+\frac{M}{2}\x)c_1+C_{4}\bb+C_{7}\x]\sum_{i} S_k^{ii}<0. 
    \end{align*}
    This is a contradiction.

    \textbf{Case 3b:} $k=n=2$. Since, $u_{nn}=u_{22}<0$, the case returns to Case $3a$.

    \textbf{Case 3c:} $k\geq 3$ or $n>k=2$ and $u_{22}>0$.
    
    In this subcase, we observe that $u_{11}\geq \frac{A+M}{8}$, and $|D^2 u|\leq M_4 (1+M)$ by Lemma \ref{lem elliptic global C^2 reduce to boundary}. If $A\geq 8$, we have $u_{11}\geq 1$, $u_{11}\geq \frac{u_{22}}{8 M_4}$, 
    and 
    \begin{align}\label{eq elliptic upper bound, subcase 3c-II+III}
        \operatorname{II}+ \operatorname{III}
        \leq &  (4\kappa_{\operatorname{max}}+ 2|\pp_u|) kS_k -\z(4\kappa_{\operatorname{max}}+ 2|\pp_u|+\frac{2}{n}\x) \sum_{u_{ii}<0}S_k^{ii}u_{ii}\notag \\
        \leq & \z[C_{6}+\z(4\ka_{\operatorname{max}}+2|\pp_u|+\frac{2}{n}\x)M_4(1+M)\x]\sum_{i} S_k^{ii}
    \end{align}
    Using Lemma \ref{lem 2.4} and setting $\delta = \frac{1}{8 M_4}$ and $\es=\frac{c_1}{10 \z(4\kappa_{\operatorname{max}}+2|D\pp|_{C^0}+ \frac{2}{n}\x)}$ in Lemma \ref{lem 2.5}, we have
    \begin{align*}
        S_{k}^{11}\geq \ttt_1(n,k,\es,\de) S_k^{nn}\geq \ttt_1 c_{n,k} \sum_{i} S_k^{ii}:=c_2 \sum_{i} S_k^{ii},
    \end{align*}
    where $c_2>0$ depends on $n$, $k$, $\ooo$, $M_4$, $|Du|_{C^0}$, $\inf f$, $|f|_{C^2}$ and $|\pp|_{C^3}$. Thus 
    \begin{align}\label{eq elliptic upper bound, subcase 3c-I}
        \operatorname{I}=-\frac{2\bb}{n}S_k^{11}u_{11}\leq -\frac{c_2 \bb}{4n}\z(A+M\x)\sum_{i} S_k^{ii}.
    \end{align}
    Substituting \eqref{eq elliptic upper bound, subcase 3c-II+III} and \eqref{eq elliptic upper bound, subcase 3c-I} into \eqref{eq elliptic upper bound, I+II+III+IV}, we obtain
    \begin{align*}
        0 \leq \sum_{i}  {S_k^{ii} P_{ii} }  \leq&  \z[C_{6}+\z(4\ka_{\operatorname{max}}+2|\pp_u|+\frac{2}{n}\x)M_4 (1+M)-\frac{\bb c_2}{4n}(A+M)\x]\sum_{i} S_k^{ii} \\
        & +\z[-\z(A+\frac{M}{2}\x)c_1+C_{4}\bb\x]\sum_{i} S_k^{ii} <0, \notag 
    \end{align*}
    if the constants $\bb$ and $A$ are chosen sufficiently large such that 
    $$
    \frac{c_2}{4n}\bb >C_{6}+\z(4\ka_{\operatorname{max}}+2|\pp|_{C^1}+\frac{2}{n}\x)M_4, \quad A>\frac{C_{4}}{c_1}\bb,
    $$
    which leads to a contradiction. 
    
    Summarizing, if the positive constants $\bb$ and $A$ satisfy the following conditions:
    $$
    \bb =\mu^{-1} \geq L_1 (M_4+1), \quad \te{and} \quad A>\max\{ 8,\bb L_1 \}, 
    $$
    for some sufficiently large constant $L_1$, then $P$ attains its minimum only on $\pa \ooo$. 
\end{proof}

\subsection{A lower estimate of the double normal derivatives on the boundary}

\begin{lemma}\label{lem elliptic u_nu nu geq -C} 
    Under the assumptions of Theorem $\ref{thm elliptic C^2 estimates}$, then we have
    \begin{align}\label{eq elliptic u_{nu nu} geq -C}
        \min_{\pa \ooo} u_{\nu \nu}  \geq - M_{6},
    \end{align}
    where the positive constant $M_{6}$ depends on $n, k, c_\ka, \Omega, M_4, |Du|_{C^0}, \inf f, |f|_{C^2}$, and $|\pp|_{C^3}$.

    In particular, when $\pp$ merely satisfies $\pp_u\leq 0$, we must require $\ooo$ to be uniformly convex. 
\end{lemma}
\begin{proof}
    Without loss of generality, assume $M:=\max\limits_{\pa \ooo}|u_{\nu \nu}|=-\min\limits_{\pa \ooo}u_{\nu \nu}>0$. Otherwise, the result follows directly from the argument in Lemma \ref{lem elliptic u_nu nu leq C}. We further assume that $u_{\nu \nu}|_{\pa \ooo}$ attains its minimum at $\tilde{z}_0$. By a similar argument, we consider the auxiliary function
    \begin{align*}
        \widetilde{P}(x) =(1+ \beta d)[D u \cdot (-D d)-\pp (x,u)] +\z(A +\frac{M}{2}\x)h(x),\quad \te{in }\ooo_\mu,
    \end{align*}
    where $h$ is defined in Lemma \ref{lem h=-d+Kd^2}, $\bb$ and $A$ are positive constants to be determined.

    Clearly, $\widetilde{P}=0$ on $\pa \ooo$. On $\pa \ooo_\mu \cap \ooo$, we have 
    \begin{align*}
        \widetilde{P}(x) \leq C_1+\mu \z[\bb C_1-\z(A+\frac{M}{2}\x)(1-K\mu)\x]<0,
    \end{align*}
    provided $A\geq \frac{8}{7}C_1 \z(\bb+\frac{1}{\mu}\x)$, with $C_1=C_1(|Du|_{C^0},|\pp|_{C^0})$.

    Next, we claim that $\widetilde{P}|_{\overline{\ooo_\mu}}$ attains its maximum only on $\pa \ooo$. then 
    \begin{align*}
        0 \leq \widetilde{P}_\nu (\tilde{z}_0)  =&  \z[u_{\nu \nu}(\tilde{z}_0) -u_i \nu_j d_{ij}- D_\nu \pp\x] + A + \frac{M} {2}  \notag\\
        \leq& \min_{\pa \ooo} u_{\nu \nu} + C_{2} + A + \frac{M} {2}=\frac{1}{2}\min_{\pa \ooo} u_{\nu \nu} + C_{2} + A,
    \end{align*}
    where $C_2>0$ depends on $\ooo$, $|Du|_{C^0}$, and $|\pp|_{C^1}$. Therefore, \eqref{eq elliptic u_{nu nu} geq -C} holds.

    To prove this claim, we assume by contradiction that $\widetilde{P}$ attains its maximum at some point  $\tilde{x}_0 \in \Omega_\mu$. Moreover, assume that $D^2 u(\tilde{x}_0)$ is diagonal. 
    All subsequent calculations are evaluated at $\tilde{x}_0$.

    Similar to the calculations in Lemma \ref{lem elliptic u_nu nu leq C}, we have
    \begin{align}\label{eq elliptic lower, 0=widetilde{P}_{i}}
        0 = \widetilde{P}_i =& \beta d_i \z[-\sum_{j} {u_j d_j } - \pp \x] + (1+ \beta d)\z[-\sum_{j} {(u_{ji} d_j + u_j d_{ji})} - \pp_{x_i}-\pp_u u_i \x] \\
        &+ \z(A + \frac{M}{2}\x)h_i,\nonumber
    \end{align}
    and
    \begin{align}\label{eq elliptic lower, 0 geq widetilde{P}_{ii}}
        0 \geq \widetilde{P}_{ii}  
        \geq&  - 2\beta u_{ii} d_i ^2  + (1+ \beta d)\z[-\sum_{j} {u_{jii} d_j }  - 2u_{ii} d_{ii} -\pp_{u} u_{ii}\x] \\
        &+ \z(A + \frac{M} {2}\x)h_{ii} - C_{3}(\ooo,|Du|_{C^0},|\pp|_{C^2})(1+\bb). \nonumber
    \end{align}
    Contacting $S_k^{ii}$ in both sides of \eqref{eq elliptic lower, 0 geq widetilde{P}_{ii}} and using $\sum_{i=1}^n S_k^{ii}\geq c_0(n,k,\inf f )$, it follows that
    \begin{align}\label{eq elliptic upper bound 0 geq I+II+III+IV}
        0 \geq \sum_{i} S_k^{ii} \widetilde{P}_{ii} 
        \geq&  - 2\beta \sum_{i} {S_k^{ii} u_{ii} d_i ^2 }  - 2(1+ \beta d)\sum_{i} {S_k^{ii} u_{ii} d_{ii} } -(1+\bb d)\pp_u \sum_{i} S_k^{ii} u_{ii} \nonumber \\
        &+ \z[\z(A + \frac{M} {2}\x)c_1  - C_{4}(1+\bb)\x]\sum_{i} {S_k^{ii} }  \notag \\
        =& \operatorname{I}+\operatorname{II}+\operatorname{III}+\operatorname{IV}, 
    \end{align}
    where the constant $C_{4}$ depends on $n, k, \ooo, |Du|_{C^0}, \inf f, |f|_{C^1}$, and $|\pp|_{C^2}$.

    We partition the indices $1\leq i\leq n$ into two categories to analyze the term $\operatorname{I}$.
    Let $B = \z\{ i:\beta d_i ^2  < \frac{1}{n},1 \leq i \leq n\x\}$, and $G = \z\{ i:\beta d_i ^2 \geq  \frac{1}{n},1 \leq i \leq n\x\}$.
    By choosing $\bb=\frac{1}{\mu} > 1$, we ensure that $G$ is nonempty. Furthermore, there exists an index $\tilde{i}_0\in G$, and without loss of generality, we may assume $\tilde{i}_0=1$
    \begin{align*}
        d_{1} ^2 \geq \frac{1}{n}.
    \end{align*}
    From \eqref{eq elliptic lower, 0=widetilde{P}_{i}}, we have
    \begin{align*}
        u_{i i}  = -\frac{{1-2Kd}}{1+ \beta d}\z(A + \frac{M}{2}\x) + \frac{{\beta \z[-\sum_{j}  {u_j d_j } - \pp \x]}}{1+ \beta d} + \frac{{-\sum_{j}  { u_j d_{ji}}  - D_i\pp }}{{d_{i } }}.
    \end{align*}
    We choose a sufficiently large $A>0$, such that for any $i \in G$
    \begin{align*}
        \z| \frac{{\beta [-\sum_{j} u_j d_j-\pp]}}{1+ \beta d} + \frac{{-\sum_{j}u_j d_{ji}- D_i\pp }}{{d_{i } }}\x| 
        \leq \z(\bb+\sqrt {n\beta}\x)C_{5} \leq \frac{A}{8},
    \end{align*}
    where $C_{5}$ depends on $\ooo$, $|Du|_{C^0}$, and $|\pp|_{C^1}$. Hence
    \begin{align}\label{eq elliptic lower bound u_{ii}<0}
        - \frac{9A}{8} -\frac{M}{2} \leq u_{i i}\leq -\frac{A+M}{8}, \quad \te{for any } i \in G.
    \end{align}

    By applying Lemma \ref{lem 2.4} and \ref{lem elliptic global C^2 reduce to boundary} together with \eqref{eq elliptic lower bound u_{ii}<0}, we derive
    \begin{align}\label{eq elliptic lower bound-I}
        \operatorname{I}=& - 2\bb \sum_{i\in G}S_k^{ii} u_{ii} d_i^2 - 2\bb \sum_{i\in B}S_k^{ii} u_{ii} d_i^2 \geq -\frac{2\bb}{n}S_k^{11}u_{11}-C_{6}(1+M)\sum_{i} S_k^{ii} \nonumber \\
        \geq & \frac{\bb}{4n}c_{n,k} \z(A+M\x)\sum_{i} S_k^{ii}-C_{6}(1+M)\sum_{i} S_k^{ii},
    \end{align}
    where $c_{n,k}=\frac{1}{n-k+2}$ and $C_{6}$ is independent of $|u|_{C^0}$. Furthermore,
    \begin{align}\label{eq elliptic lower bound-II+III}
        \operatorname{II}+\operatorname{III}\geq -C_{7}(1+M)\sum_{i} S_k^{ii}\quad \te{and} \quad 
        \operatorname{IV}\geq (Ac_1-2C_{4}\bb)\sum_{i} S_k^{ii}.
    \end{align}

    Thus, substituting \eqref{eq elliptic lower bound-I} and \eqref{eq elliptic lower bound-II+III} into \eqref{eq elliptic upper bound 0 geq I+II+III+IV}, we obtain
    \begin{align*}
        0  
        \geq&  \z[\frac{\beta c_{n,k}}{4n}-C_{6}-C_7\x] \z(A+M\x)\sum_{i} S_k^{ii } +(Ac_1-2 C_{4}\bb) \sum_{i} S_k^{ii}   
        > 0,  
    \end{align*}
    provided $\bb>\frac{4 n(C_{6}+C_7)}{c_{n,k}}$ and $A>\max\z\{\frac{2 C_{4}\bb}{c_1},1\x\}$.

    In summary, if the positive constants $\bb$ and $A$ satisfy the following conditions:
    $$
    \bb=\frac{1}{\mu} >\frac{4 n(C_{6}+C_7)}{c_{n,k}} \quad \te{and} \quad A>\max\z\{1, \frac{16}{7}\bb C_1, \frac{2 C_{4} \bb}{c_1}, 8\z(\bb+\sqrt{n\bb}\x)C_{5}\x\},
    $$
    then $\widetilde{P}$ attains its maximum only on $\pa \ooo$. This completes the proof. 
\end{proof}

\begin{proof}[Proof of Theorem $\ref{thm elliptic C^2 estimates}$]
    Combining Lemmas \ref{lem elliptic global C^2 reduce to boundary}, \ref{lem elliptic u_nu nu leq C}, and \ref{lem elliptic u_nu nu geq -C}, we complete the proof of Theorem \ref{thm elliptic C^2 estimates}.
\end{proof}

\section{Existence of the Boundary Problems}\label{sec-existence}
\setcounter{equation}{0}

In this section, we complete the proof of the Theorems \ref{thm 1.2-Neumann for elliptic sum} and \ref{thm 1.3-classical Neumann for elliptic sum}.

\subsection{Proof of Theorem \ref{thm 1.2-Neumann for elliptic sum}}

For the Neumann boundary value problem of the sum Hessian equations \eqref{eq-Neumann for elliptic sum (thm 1.2)}, we have derived a global $C^2$ estimates in Sections \ref{sec 3-C^0 and C^1} and \ref{sec 4-second derivative}. Therefore, the sum Hessian equations \eqref{eq-Neumann for elliptic sum (thm 1.2)} are uniformly elliptic in $\overline{\ooo}$. By Lemma \ref{lem 2.2} $(b)$, it follows that $\z(S_k\x)^{\frac{1}{k}}$ is concave in $\tilde{\gaa}_{k}$, Consequently, we obtain the following global H\"{o}lder estimate (see \cite[Theorem 1.1]{Lieberman-Trudinger-1986-TAMS}):
\begin{equation*}
    |u|_{C^{2,\ga}(\overline{\ooo})}\leq C,
\end{equation*}
where the positive constants $C$ and $\ga$ depend on $n$, $k$, $c_{\pp}$, $a_{\ka}$, $c_{\ka}$, $\ooo$, $\inf f$, $|f|_{C^2}$, and $|\pp|_{C^3}$. Finally, employing the method of continuity (see \cite[Theorem 17.28]{GT}), we establish the existence result, thereby completing the proof of Theorem \ref{thm 1.2-Neumann for elliptic sum}.

\subsection{Proof of Theorem \ref{thm 1.3-classical Neumann for elliptic sum}}
    Assume $u_\es$ is the $(k-1)$-admissible solution to \eqref{eq elliptic perturbed Neumann problem}. 
    Define $w_\es=u_\es-\frac{1}{|\ooo|}\int_{\ooo} u_\es dx$. Then $w_\es$ is a $(k-1)$-admissible solution to
    \begin{equation}\label{eq subsec 5.2}
        \left\{\begin{aligned}
            & \si_k(D^2 w_\es)+\al \si_{k-1}(D^2 w_\es)=f(x) && \te{in } \ooo, \\
            & (w_\es)_\nu= -\es w_\es-\frac{1}{|\ooo|}\int_{\ooo} \es u_\es dx+\pp(x) && \te{on } \pa \ooo.
        \end{aligned}\right.
    \end{equation}

    Following the proof of Theorem \ref{thm elliptic C^0}, we obtain $|\es u_\es|_{C^0(\overline{\ooo})}\leq \overline{M}_0$. By the uniform convexity of $\ooo$ and adapting the argument of \cite[Proposition 5]{Qiu-Xia-IMRN-2019}, we derive 
    $$|Du_\es|_{C^0(\overline{\ooo})} + \z|w_\es \x|_{C^0(\overline{\ooo})}\leq \overline{M}_1,
    $$
    
    Applying Theorem \ref{thm elliptic C^2 estimates}, we further obtain $|D^2 u_\es|_{C^0(\overline{\ooo})}\leq \overline{M}_2$. Crucially, the constants $\overline{M}_0$, $\overline{M}_1$ and $\overline{M}_2$ are independent of $\es$ and $|u_\es|_{C^0}$. The existence of a solution to \eqref{eq subsec 5.2} then follows from Schauder estimates \cite{Lieberman-Trudinger-1986-TAMS} and standard elliptic theory. 
    
    By passing to a subsequence (still denoted by $\es \to 0$), we find that $w_\es \to \bar{u}$ in $C^2(\overline{\ooo})$ and $-\es w_\es-\frac{1}{|\ooo|}\int_{\ooo} \es u_\es \to s$ for some $s\in \mathbb{R}$. The pair $(s,\bar{u})$ thus solves \eqref{eq-classical Neumann for elliptic sum (thm 1.3)}. The uniqueness follows from the maximum principle and Hopf's lemma, see \cite{Qiu-Xia-IMRN-2019} for details.

\section{Neumann problem for parabolic sum Hessian equation}\label{sec-parabolic}
\setcounter{equation}{0}

In this section, we investigate the Neumann problem for the parabolic sum Hessian equation \eqref{eq parabolic sum hessian-1} and prove Theorem \ref{thm 1.5-parabolic}. First, we denote
$$
\widetilde{F}(D^2 u)=\log S_k(D^2 u ),\ \tilde{f}=\log f,\ \widetilde{F}^{ij}(D^2 u)=\frac{\pa \widetilde{F}(D^2 u )}{\pa u_{ij}}, \te{ and } \widetilde{F}^{ij,pq}=\frac{\pa^2 \widetilde{F}(D^2 u)}{\pa u_{ij}\pa u_{pq}}.
$$

\subsection{$u_t$ and $C^0$ estimates}

By applying the classical maximum principle and Hopf's lemma for parabolic equations, we derive the following estimate for $u_t$, which consequently yields the $C^0$ estimate. See \cite{chen-ma-zhang-2021} for details.
\begin{lemma}[Lemma 3.1 in \cite{chen-ma-zhang-2021}]\label{lem parabolic u_t estimate}
    Let $\Omega \subset \mathbb{R}^n$ be a $C^3$ domain, and let $f,\pp
    $ satisfy \eqref{condition f>0} and $\pp_u\leq 0$. Suppose $u \in C^3(\overline{\ooo}\times[0,T))$ is a $(k-1)$-admissible solution to \eqref{eq parabolic sum hessian-1}. Then 
    \par
    \noindent 
    $(1)$ 
    $\min\z\{\min\limits_{\overline{\ooo}} u_t(\cdot,0),0\x\} \leq u_t(x,t)\leq  \max\z\{\max\limits_{\overline{\ooo}}u_t(\cdot,0),0\x\}, \ \forall \ (x,t)\in \ooo\times (0,T)$.
    \par
    \noindent 
    $(2)$ 
    If $f$ satisfies \eqref{condition c_f}, then for any $0\leq r\leq  c_f$, we have
    \begin{equation}\label{eq parabolic u_t e^{r t}}
        \min\z\{\min_{\overline{\ooo}} u_t(\cdot,0),0\x\}\leq u_t(x,t)e^{r t}\leq \max\z\{\max_{\overline{\ooo}}u_t(\cdot,0),0\x\}, \ \forall \ (x,t)\in \overline{\ooo}\times (0,T).
    \end{equation}
    \par
    \noindent 
    $(3)$
    If \eqref{condition S_k(D^2 u_0)} holds, then $u_t \equiv 0$ or $u_t(x,t)>0$ for all $(x,t)\in \overline{\ooo}\times (0,T)$.
\end{lemma}

\begin{lemma}[Theorem 4.1 in \cite{chen-ma-zhang-2021}]\label{lem parabolic C^0}
    Let $\ooo\subset \rr^n$ be a $C^3$ domain, and let $f,\pp 
    $ satisfy \eqref{condition f>0} and \eqref{condition c_pp}. Suppose $u\in C^3(\overline{\ooo}\times [0,T))$ is a $(k-1)$-admissible solution to \eqref{eq parabolic sum hessian-1}. Furthermore, if $f$ satisfies \eqref{condition c_f} or $u_0$ satisfies \eqref{condition S_k(D^2 u_0)}, then 
    \begin{equation*}
        |u(x,t)|\leq N_0,\quad \te{for any }  (x,t)\in \ooo\times [0,T),
    \end{equation*}
    where the constant $N_0>0$ depends on $c_f$, $c_\pp$, $|u_0|_{C^0}$, $|u_t(\cdot,0)|_{C^0}$, and $|\pp(\cdot,0)|_{C^0}$.
\end{lemma}

\subsection{Gradient estimate}

We denote $N_0=\sup_{\ooo\times [0,T)} |u|$, $N_1=\sup_{\ooo\times[0,T)}|u_t|$. It follows from \eqref{eq parabolic sum hessian-1} that
\begin{align*}
    0<\tilde{c}_0=:e^{-N_1} \cdot\inf_{\overline{\ooo}\times [-N_0,N_0]} f <S_k(D^2 u)=f(x,u)e^{u_t}\leq  e^{N_1} \cdot\sup_{\overline{\ooo}\times[-N_0,N_0]}f:=\tilde{c}_1.
\end{align*}
We now proceed to establish a global gradient estimate. The proof consists of two steps: first, deriving an interior gradient estimate, and second, obtaining a gradient estimate near the boundary.

\begin{theorem}\label{thm parabolic interior gradient}
    Let $\ooo\subset \rr^n$ be a $C^3$ domain, $f>0$. If $u\in C^3(\overline{\ooo}\times [0,T))$ is a $(k-1)$-admissible solution to \eqref{eq parabolic sum hessian-1}. Then, we have
    \begin{equation*}
        \sup_{(\ooo\xg \ooo_\mu)\times [0,T)}
        |Du|\leq N_2, 
    \end{equation*}
    where $N_2>0$ depends on $n$, $k$, $\mu$, $\tilde{c}_0$, $\tilde{c}_1$, $N_0$, $N_1$, $|Du_0|_{C^0}$, $\inf f$, $|f|_{C^0}$, and $|D_xf|_{C^0}$. 
\end{theorem}
\begin{proof}
    The proof follows a similar approach to that of Theorem \ref{thm elliptic interior gradient}. Below, we provide a sketch of the proof.

    For any $z_0 \in \ooo\xg \ooo_{\mu}$ and $T_0\in (0,T)$, we  introduce the auxiliary function
    \begin{equation*}
        G(x,t)=|Du| \p(u) \rho(x), \quad \te{in } B_\mu(z_0) \times [0,T_0],
    \end{equation*}
    where $\p(u)=(3N_0-u)^{-\frac{1}{2}}$ and $\rho(x)=\mu^2-|x-z_0|^2$. 
    Assume $G$ attains its maximum at $(x_0,t_0)\in B_\mu(z_0)\times [0,T_0]$. Without loss of generality, we may assume $t_0>0$. Otherwise, 
    $$
    |Du(z_0,t)|\leq\sqrt{2} |Du(x_0,0)|\leq \sqrt{2}|Du_0|_{C^0}\quad \te{for all } t\in[0,T_0].
    $$
    
    We choose the coordinates such that $u_1=|Du|>0$ and $\{u_{ij}\}_{2\leq i,j\leq n}$ is diagonal at $(x_0,t_0)$. 
    Under this coordinate system, the function
    $$
    \PPP(x,t)=\log u_1+\log\p(u)+\log \rho(x)
    $$
    attains a local maximum at $(x_0,t_0)$. All subsequent calculations are evaluated at $ (x_0,t_0)$. Then 
    $$
    0= \PPP_i =  \frac{u_{1i}}{u_1}+\frac{\p'}{\p}u_i +\frac{\rho_i}{\rho} \quad \te{and} \quad
    0\leq \PPP_t=  \frac{u_{1t}}{u_1}+\frac{\p'}{\p} u_t.
    $$
    Using $\widetilde{F}^{ij}u_{ij1}-u_{t1}=\tilde{f}_{x_1}+\tilde{f}_u u_1$ and $\frac{\p'}{\p}\widetilde{F}^{ij}u_{ij}\geq -\frac{c_{n,k}}{N_0}\sum_{i=1}^n\widetilde{F}^{ii}$ (Lemma \ref{lem 2.6}), we have
    \begin{align} 
        0\geq  \widetilde{F}^{ij} \PPP_{ij}-\PPP_t= & \frac{1}{u_1}(\widetilde{F}^{ij} u_{1ij}-u_{1t})-\widetilde{F}^{ij} \frac{u_{1i}u_{1j}}{u_1^2}+\z[\frac{\p''}{\p} -\frac{(\p')^2}{\p^2}\x]\widetilde{F}^{11} u_1^2 \nonumber \\
        &  +\frac{\widetilde{F}^{ij}\rho_{ij}}{\rho}-\widetilde{F}^{ij} \frac{\rho_i \rho_j}{\rho^2} +\frac{\p'}{\p}(\widetilde{F}^{ij} u_{ij}-u_t)\nonumber \\
        \geq & \frac{\tilde{f}_{x_1}}{u_1}+\tilde{f}_u+\frac{\widetilde{F}^{11}u_1^2}{64 N_0^2}-\z[\frac{u_1}{2N_0}\frac{\mu}{\rho}+\frac{2}{\rho}+8\frac{\mu^2 }{\rho^2 }+\frac{c_{n,k}}{N_0}  \x]\sum_{i=1}^n\widetilde{F}^{ii} -\frac{|u_t|}{4N_0}. \label{eq 6.2}
    \end{align}

    We may assume $|Du(z_0,t_0)|\geq 24\sqrt{2}\frac{N_0}{\mu}$, this implies that $u_{11}<0$. Following the arguments in Theorem \ref{thm elliptic interior gradient} and applying Lemma \ref{lem 2.3}, we obtain $\widetilde{F}^{11}\geq c_{n,k}\sum_{i} \widetilde{F}^{ii}$ and
    \begin{equation}\label{eq parabolic sum geq tilde c_2} 
        \sum_{i} \widetilde{F}^{ii} \geq \frac{1}{\tilde{c}_1}c_0(n,k,\tilde{c}_0):=\tilde{c}_2(n,k,\tilde{c}_0,\tilde{c}_1)>0,
    \end{equation}
    where $\tilde{c}_0\leq S_k\leq \tilde{c}_1$.
    Consequently, \eqref{eq 6.2} yields the inequality
    \begin{align*}
        0\geq -\frac{|\tilde{f}_{x_1}|}{u_1} -|\tilde{f}_u|-\frac{|u_t|}{4N_0}+\z[\frac{c_{n,k}}{N_0^2}u_1^2-\frac{\mu}{2N_0}\frac{u_1}{\rho}-\frac{c_{n,k}}{N_0} -\frac{2}{\rho}-8\frac{\mu^2 }{\rho^2 } \x]\tilde{c}_2.
    \end{align*}
    This implies that $\rho(x_0)u_{1}(x_0,t_0)\leq C$, and consequently, $\sup_{t\in [0,T_0]}|Du(z_0,t)|\leq N_2$.
    Since $z_0\in \ooo\xg \ooo_{\mu}$ and $T_0\in(0,T)$ are arbitrary, and $N_2$ is independent of $z_0$ and $T_0$, the proof is complete.
\end{proof}

In the following, we also prove the near boundary gradient estimate for oblique boundary value problem.
\begin{theorem}\label{thm parabolic near boundary gradient}
    Let $\ooo\subset \rr^n $ be a $C^3$ domain, $f>0$. If $u \in C^3(\overline{\ooo}\times [0,T))$ is the $(k-1)$-admissible solution to sum Hessian equation 
    \begin{equation*}
        \z\{\begin{aligned}
            & u_t=\log S_k(D^2u)- \log f(x,u)&&  \te{in }\ooo \times[0,T), \\
            & u_\bb=\pp(x,u)&&  \te{on }\pa \ooo\times[0,T), \\
            & u(x,0)=u_0(x) && \te{in } \ooo,
        \end{aligned}\x.
    \end{equation*}
    where $\bb\in C^{3}(\pa \ooo)$ is a unit vector field of $\pa \ooo$ with $\z<\bb,\nu\x>\geq \de_0$ for some $\de_0>0$, $\nu$ is the outer unit normal vector field of $\pa \ooo$. Then there exists a small positive constant $\mu$, depending on $n,k,\de_0, \ooo, \tilde{c}_0, \tilde{c}_1, N_0, N_1, |Du_0|_{C^0}, \inf f, |f|_{C^1},|\pp|_{C^3}$, and $|\bb|_{C^3}$, such that
    \begin{equation}\label{eq parabolic near boundary gradient}
        \sup_{\ooo_\mu\times [0,T)}
        |Du|\leq \max \{N_2,N_3\}, 
    \end{equation}
    where $N_2$ is defined in Theorem $\ref{thm parabolic interior gradient}$, the positive constant $N_3$ depends on $n$, $k$, $\mu$, $\de_0$, $\ooo$, $\tilde{c}_0$, $\tilde{c}_1$, $N_0$, $N_1$, $|Du_0|_{C^0}$, $\inf f$, $|f|_{C^1}$, $|\pp|_{C^3}$, and $|\bb|_{C^3}$. 
\end{theorem}
\begin{proof}
    The proof is similar to Theorem \ref{thm elliptic near boundary gradient}, with some modifications. 

    We extend the unit vector field $\bb$ smoothly to $\ooo_\mu\times [0,T)$ such that $\z<\bb,-Dd\x>=\cos \ttt\geq \de_0$. Let 
    $$w=u+\frac{\pp d}{\cos \ttt}\quad \te{and} \quad \p=\sum_{i,j}C^{ij}w_i w_j \triangleq \sum_{i,j}(\de_{ij}-d_i d_j)w_i w_j.
    $$ 
    For any $T_0\in [0,T)$, we define
    \begin{equation*}
        \varPhi =\log \p+g(u)+\al_0 d, \quad \te{in } \ooo_\mu\times[0,T_0],
    \end{equation*}
    where $g(u)=-\frac{1}{2} \log(1+M_0-u)$, and $\mu,\al_0>0$ are constants to be determined later. 

    \textbf{Step 1.}
    First, we establish the following claim:
    \begin{equation}\label{eq parabolic global gradient to boundary}
        \sup_{\ooo\times [0,T)} |Du|\leq C_1 (1+\sup_{\pa \ooo\times[0,T)}|Du|),
    \end{equation} 
    where $C_1$ depends on $n$, $k$, $|Du_0|_{C^0}$, $N_0$, $N_1$, $\inf f$, and $|f|_{C^1}$. To prove this ,claim, we choose the auxiliary function $\chi=\log |Du|^2+g(u)$ in $\overline{\ooo} \times [0,T_0]$, where $T_0\in(0,T)$ is arbitrary. Suppose $\chi$ attains its maximum over $\overline{\ooo}\times [0,T_0]$ at $(z_0,s_0)$. If $z_0\in \pa \ooo$ or $s_0=0$, then \eqref{eq parabolic global gradient to boundary} follows directly. Without loss of generality, we may assume $(z_0,s_0)\in \ooo \times (0,T_0]$, $D^2 u$ is diagonal and $u_1^2\geq \frac{1}{n}|Du|^2$ at $(z_0,s_0)$. At $(z_0,s_0)$, we have 
    \begin{align*}
        0=\chi_i =& 2\frac{u_i u_{ii}}{|Du|^2}+g' u_i, \quad  0\leq \chi_t = 2\frac{u_l u_{lt}}{|Du|^2}+g' u_t.
    \end{align*}
    Consequently, $u_{11}<0$, which implies $\widetilde{F}^{11}\geq c_{n,k}\sum_i \widetilde{F}^{ii}$.
    Applying Lemma \ref{lem 2.6}, we obtain
    \begin{align*}
        0\geq \widetilde{F}^{ii} \chi_{ii} -\chi_t= & 2\frac{\widetilde{F}^{ii}u_{ii}^2}{|Du|^2}+\frac{2u_l}{|Du|^2}(\widetilde{F}^{ii}u_{iil}-u_{tl})+[g''-(g')^2]\widetilde{F}^{ii}u_{i}^2+g' (\widetilde{F}^{ii}u_{ii}-u_t)\\
        \geq & -\frac{2|f_{x_l}|}{|Du|f}-2\frac{|f_u|}{f}-\frac{|u_t|}{2}+\frac{c_{n,k}u_1^2}{4(1+2N_0)^2}\sum_i \widetilde{F}^{ii}-c_{n,k}\sum_i \widetilde{F}^{ii}.
    \end{align*}
    Therefore, \eqref{eq parabolic global gradient to boundary} holds by \eqref{eq parabolic sum geq tilde c_2} and the arbitrariness of $T_0$. Combining \eqref{eq 3.9}, \eqref{eq 3.10} and \eqref{eq parabolic global gradient to boundary} yields $\sup_{\ooo} |Du|^2 \leq C_2(1+\sup_{\pa \ooo}\p)$.

    \textbf{Step 2.}
    From the above argument, it suffices to prove a bound for $\p$ or $|Dw|$.

    Assume that $\PP$ attains ite maximum over $\overline{\ooo_\mu}\times[0,T_0]$ at $(x_0,t_0)$. If $t_0=0$, we may choose $\mu>0$ small such that \eqref{eq parabolic near boundary gradient} holds. Thus, we assume $t_0>0$. By Theorem \ref{thm parabolic interior gradient}, it suffices to consider the following two cases.

    \textbf{Case 1:} $x_0 \in \pa \ooo$.
    The proof of this case relies solely on the Neumann boundary condition, and $\al_0=\al_0(\pa \ooo,\de_0,|\bb|_{C^2},|\pp|_{C^0})>0$ has been chosen appropriately. Thus, we omit the proof here, and the detailed proof can be found in \cite{chen-ma-zhang-2021}.

    \textbf{Case 2:} $x_0 \in \ooo_\mu$. Let $T_j=\sum_i C^{ij}w_i$ and define $\mathcal{T}=(T_1,\cdots,T_n)$.
    By \eqref{eq parabolic global gradient to boundary}, we may assume that the maximum of $|Du|$ on $\pa \ooo\times[0,T_0]$ is attained at $(x_1,t_1)$. 
    Furthermore, we assume that both $|Du(x_0,t_0)|$ and $|Du(x_1,t_1)|$ are sufficiently large and $D^2 u(x_0,t_0)$ is diagonal, ensuring the following conditions hold (similar to Theorem \ref{thm elliptic near boundary gradient}):
    \begin{align*}
        & |Du(x_1,t_1)| \te{ is comparable to } |Dw(x_1,t_1)|, \quad \te{and} \quad T_1(x_0,t_0)>0, \\
        & w_1, \ u_1, \ T_1, \ |Dw|, \ |Du|, \ |\mathcal{T}|  \te{ and } \p^\frac{1}{2} \te{ are mutually comparable at } (x_0,t_0).
    \end{align*}
    All subsequent calculations are evaluated at $ (x_0,t_0)$. Through direct calculation, we obtain
    \begin{equation}\label{eq 6.5}
        0=\PP_i =\frac{\phi_i}{\phi} + g' u_i+\al_0 d_i\,\quad \te{and} \quad 0\geq \PP_t =\frac{\p_t}{\p}+g'(u)u_t.
    \end{equation}
    Using these relations, the Cauchy inequality and $\sum_i \widetilde{F}^{ii}u_{ii}\geq -c_{n,k}\sum_i \widetilde{F}^{ii}$, we obtain
    \begin{align}
        0\geq&  \widetilde{F}^{ii} \PP_{ii}-\PP_t \nonumber \\
        \geq & \z[\widetilde{F}^{ii} \frac{\p_{ii}}{\p}  -\frac{\p_t}{\p}\x]+\z[g''-\frac{3}{2}(g')^2\x] \widetilde{F}^{ii} u_i^2+g'(\widetilde{F}^{ii} u_{ii}-u_t)-2 \al_0^2 \widetilde{F}^{ii} d_i^2+\al_0 \widetilde{F}^{ii} d_{ii} \nonumber \\
        \geq & \z[\widetilde{F}^{ii} \frac{\p_{ii}}{\p}-\frac{\p_t}{\p}\x]+\z[\frac{1}{8(1+2N_0)^2}\sum_i \widetilde{F}^{ii}u_i^2-C_3(1+\al_0+\al_0^2)\sum_{i} \widetilde{F}^{ii}-\frac{|u_t|}{2}\x] \nonumber\\
        =& \operatorname{I}+\operatorname{II}. \label{eq parabolic thm 6.6 0 geq I+II+III+IV}
    \end{align}

    Following similar arguments as in Theorem \ref{thm elliptic near boundary gradient}, we can show that $u_{11}<0$ if $\mu>0$ is sufficiently small.
    Thus, by Lemma \ref{lem 2.4} and \eqref{eq parabolic sum geq tilde c_2}, we obtain
    \begin{equation}\label{eq parabolic F^11 geq sum F^ii}
        \widetilde{F}^{11} \geq c_{n,k} \sum_{i} \widetilde{F}^{ii}\geq c_{n,k}\tilde{c}_2(n,k, \tilde{c}_0, \tilde{c}_1)>0.
    \end{equation}

    Next, we proceed to estimate the following term:
    \begin{align*}
        \operatorname{I}=&  \frac{1}{\p}\sum\limits_{i,p,q } \widetilde{F}^{ii} {C^{pq}}_{,ii} w_p w_q+\frac{2}{\p}\sum_{i,p,q } C^{pq}w_q (\widetilde{F}^{ii} w_{iip}-w_{tp}) \\
        & +\frac{4}{\p} \sum_{i,p,q} {\widetilde{F}^{ii} {C^{pq}}_{,i} w_{ip} w_q}+\frac{2}{\p}\sum_{i,p,q} \widetilde{F}^{ii} C^{pq} w_{ip} w_{iq}\\
        =& \operatorname{I}_1 + \operatorname{I}
        _2+\operatorname{I}_3+\operatorname{I}_4.
    \end{align*}
    It is suffices to estimate $\operatorname{I}_2$, as the other terms can be derived from Theorem \ref{thm elliptic near boundary gradient}. Combining \eqref{eq 3.22}, \eqref{eq 3.23}, \eqref{eq parabolic F^11 geq sum F^ii}, and
    \begin{align*}
        \z(\frac{\pp d}{\cos \ttt}\x)_{tp}=\frac{d}{\cos \ttt}[\pp_u u_{tp}+\pp_{uu} u_t u_p+\pp_{x_p u}u_t]+\z(\frac{d}{\cos \ttt}\x)_p \pp_u u_t,
    \end{align*}
    we obtain
    \begin{align*}
        \p \operatorname{I}_2
        \geq & \z(1+\frac{\pp_u d}{\cos \ttt}\x) T_p(\widetilde{F}^{ii}u_{iip}-u_{tp}) -C_4(|Dw|^3+d|Dw|^4)\z(\sum_i \widetilde{F}^{ii}+1\x) \nonumber\\
        &- C_4(|Dw|+d|Dw|) \sum_i \widetilde{F}^{ii}|u_{ii}|\nonumber\\
        \geq & -C_5(|Dw|^3+d|Dw|^4)\sum_i \widetilde{F}^{ii}-C_5(|Dw|+d|Dw|^2)\sum_i \widetilde{F}^{ii}|u_{ii}|.
    \end{align*} 
    Therefore, we establish the following inequality, analogous to \eqref{eq 3.29}:
    \begin{align}\label{eq parabolic boundary gradient-I}
        \p \operatorname{I}\geq  -C_6(|Dw|^3+d|Dw|^4)\sum_i \widetilde{F}^{ii}.
    \end{align} 
    Substituting \eqref{eq parabolic F^11 geq sum F^ii} and \eqref{eq parabolic boundary gradient-I} into \eqref{eq parabolic thm 6.6 0 geq I+II+III+IV}, we obtain
    \begin{align*}
        0\geq& \z[\z(\frac{c_{n,k}}{8(1+2N_0)^2}- C_6 d\x)|Dw|^2-C_6|Dw|-C_7 \x] \sum_{i} \widetilde{F}^{ii}-\frac{|u_t|}{2}\\
        \geq & \z[\frac{c_{n,k}}{16(1+2N_0)^2}|Dw|^2-C_6|Dw|-C_8 \x] \tilde{c}_2.
    \end{align*} 
    The proof is complete if we choose $\mu>0$ sufficiently small.
\end{proof}

\subsection{Second-order derivatives estimate}

\begin{theorem}\label{thm parabolic C^2}
    Let $\ooo\subset \rr^n$ be a $C^4$ uniformly convex domain, and let $f,\pp$ satisfy $f>0$ and $\pp_u\leq 0$. If $u\in C^4(\overline{\ooo}\times [0,T))$ is a $(k-1)$-admissible solution to \eqref{eq parabolic sum hessian-1}. Then 
    \begin{equation*}
        \sup_{\ooo\times [0,T)}|D^2 u|\leq N_4,
    \end{equation*}
    where $N_4>0$ depends on $n$, $k$, $\ooo$, $\tilde{c}_1$, $\tilde{c}_2$, $N_1$, $|D^2 u_0|_{C^0}$, $|Du|_{C^0}$, $\inf f$, $|f|_{C^2}$, and $|\pp|_{C^3}$. 
    
    In particular, if $\pp$ satisfies \eqref{condition c_pp}, the uniformly convexity on $\ooo$ can be weakened to almost convex and uniformly $(k-1)$-convex with $a_\ka>\frac{1}{2}c_{\pp}$. In this case, the constant $N_4$ additionally depends on $a_\ka$, $c_{\pp}$ and $c_\ka$.
\end{theorem}

\begin{lemma}\label{lem parabolic global C^2 reduce to boundary}
    Let $\ooo\subset \rr^n$ be a $C^4$ uniformly convex domain, and let $f,\pp$ satisfy $f>0$ and $\pp_u\leq 0$. If $u\in C^4(\overline{\ooo}\times [0,T))$ is a $(k-1)$-admissible solution to \eqref{eq parabolic sum hessian-1}. Then  
    \begin{equation*}
        \sup_{\ooo\times [0,T)}|D^2u |\leq N_5\z(1+\sup_{\pa \ooo\times [0,T)}|u_{\nu\nu}|\x),
    \end{equation*}  
    where $N_5>0$ depends on $n$, $k$, $\ooo$, $\tilde{c}_1$, $\tilde{c}_2$, $|D^2 u_0|_{C^0}$, $|Du|_{C^0}$, $\inf f$, $|f|_{C^2}$, and $|\pp|_{C^3}$. 
    
    In particular, if $\pp$ satisfies \eqref{condition c_pp}, the uniformly convexity can be weakened to almost convex with $a_\ka>\frac{1}{2}c_{\pp}$. In this case, the constant $N_5$ additionally depends on $a_\ka$ and $c_{\pp}$.
\end{lemma}
\begin{proof}
    For any $T_0\in [0,T)$, we introduce the following auxiliary function
    \begin{align*}
        v(x, t,\xi)= u_{\xi \xi}-v'(x,t,\xi)+K_1 |x|^2+K_2|D u|^2,\quad \te{in } \ooo\times [0,T_0] \times\mathbb{S}^{n-1},
    \end{align*}
    where $K_1,K_2>0$ are two constants to be determined later, and
    \begin{align*}
        & v' (x,t, \xi) = 2 (\xi \cdot \nu) \xi'^i(D_i \pp - u_l D_i \nu^l)= a^l u_l +b, \quad  \xi' = \xi - (\xi \cdot \nu) \nu, \\
        & a^l = 2(\xi \cdot \nu) ( \xi'^l \pp_u -\xi'^i D_i \nu^l), \quad  b = 2 (\xi \cdot \nu) \xi'^i  \pp_{x_i}.
    \end{align*} 

    For any fixed $\xi \in \mathbb{S}^{n-1}$, we may assume that $v(\cdot,\xi)$ attains its maximum over $\overline{\ooo}\times [0,T_0]$ at the point $(x_0,t_0)$ satisfying
    $$
    v(x_0,t_0,\xi)>0,\quad u_{\xi\xi}(x_0,t_0)>0, \quad \te{and}\quad t_0>0.
    $$ 
    Otherwise, $\sup_{\overline{\ooo}\times[0,T]}u_{\xi\xi}
    \leq C_1(|D^2 u_0|_{C^0}+K_1+1)$. 
    
    In what follows, we assume $(x_0,t_0)\in \ooo\times (0,T_0]$. We further assume that $D^2 u$ is diagonal at $(x_0,t_0)$ and $u_{11}\geq\cdots \geq u_{nn}$. All subsequent calculations are evaluated at $(x_0,t_0)$. At $(x_0,t_0)$, we have  
    \begin{align*}
        0=v_i=u_{i\xi\xi}-v'_i+2K_1 x_i+2K_2 u_l u_{li}, \quad 0\leq v_t=u_{t\xi\xi}-v'_t +2K_2 u_l u_{lt},   
    \end{align*} 
    and
    \begin{align}
        0\geq \widetilde{F}^{ii} v_{ii}-v_t  = &(\widetilde{F}^{ii} u_{ii\xi \xi }-u_{t\xi\xi})-(\widetilde{F}^{ii} v'_{ii}-v'_t) +2K_2 u_l (\widetilde{F}^{ii}u_{iil}-u_{tl}) \label{eq parabolic 0 geq F^ii v_ii-v_t-I}\\
        &+2K_1 \sum_i \widetilde{F}^{ii}+2K_2 \sum_i \widetilde{F}^{ii}u_{ii}^2.\nonumber
    \end{align}

    Since $S_k^{\frac{1}{k}}$ is concave, so is $\widetilde{F}$. Differentiating \eqref{eq parabolic sum hessian-1} twice yields
    $$
    2K_2 u_l (\widetilde{F}^{ii}u_{lii}-u_{lt})=2K_2 u_l(\tilde{f}_{x_l}+\tilde{f}_u u_l)\geq -C_2K_2,
    $$
    \begin{align*}
        \widetilde{F}^{ii}u_{ii\xi\xi}-u_{t\xi\xi}\geq \widetilde{F}^{ij,pq}u_{ij\xi}u_{pq\xi}+\widetilde{F}^{ii}u_{ii\xi\xi}-u_{t\xi\xi}=\tilde{f}_{\xi\xi}\geq -C_{3}+\tilde{f}_u u_{\xi \xi}.
    \end{align*}
    By direct calculation, we obtain 
    \begin{align*}
        -(\widetilde{F}^{ii}v'_{ii}-v'_t)=& -a^l(\widetilde{F}^{ii}u_{lii}-u_{lt})-\widetilde{F}^{ii}(2D_ia^l u_{li}+D_{ii}a^l u_l +b_{ii})+a^l_t u_l+b_t\\
        \geq & -C_{4}-C_{4} \sum_i \widetilde{F}^{ii}|u_{ii}|-C_{4}\sum_{i}\widetilde{F}^{ii}.
    \end{align*}
    Without loss of generality, we may assume that $u_{11}$ is sufficiently large. Then by using \eqref{eq 2.1} and choosing $K_2$ sufficiently large, we have
    $$
    K_2 \sum_i \widetilde{F}^{ii}u_{ii}^2-C_2 K_2 -C_3-C_4 +\tilde{f}_u u_{\xi \xi}\geq \ttt_0 K_2 u_{11}-C_5(K_2+1+u_{\xi\xi})\geq 0,
    $$
    where the constant $\ttt_0$ is defined in Lemma \ref{lem 2.2}.

    Substituting these estimates into \eqref{eq parabolic 0 geq F^ii v_ii-v_t-I} and choosing $K_1$ sufficiently large, we deduce
    \begin{align*}
        0\geq \sum_{i}\widetilde{F}^{ii}v_{ii}-v_t\geq  K_2 \sum_{i} \widetilde{F}^{ii}\z(|u_{ii}| -\frac{C_{4}}{2K_2}\x)^2+\z(2K_1-\frac{C_{4}^2}{4K_2}-C_4\x)\sum_{i} \widetilde{F}^{ii}>0,
    \end{align*}
    which lead to a contradiction. 
    Therefore, $v|_{\overline{\ooo}\times[0,T_0]\times \mathbb{S}^{n-1}}$ attains its maximum at $(x_0,t_0,\xi_0)\in \pa \ooo\times (0,T_0]\times \mathbb{S}^{n-1}$.

    Finally, we consider two cases: $(1)$ $\xi_0$ is tangent to $\pa \ooo$, and $(2)$ $\xi_0$ is not tangent to $\pa \ooo$. The argument is the same as the proof of \cite[Lemma 6.2]{chen-ma-zhang-2021}. It is worth noting that in case $(1)$, we only require $\pp_u-2D_1\nu^1\leq c_{\pp}-\min_i \ka_i <0$. 
    The details are omitted here.
\end{proof}

\begin{lemma}\label{lem parabolic lower bound of boundary C^2}
    Under the assumptions of Lemma $\ref{thm parabolic C^2}$, we have
    \begin{equation}\label{eq parabolic u_{nu nu} geq -C}
        \inf_{\pa \ooo\times[0,T)}u_{\nu\nu} \geq -N_6,
    \end{equation}
    where $N_6>0$ depends on $n$, $k$, $\ooo$, $\tilde{c}_1$, $\tilde{c}_2$, $N_1$, $N_5$, $|u_0|_{C^2}$, $|Du|_{C^0}$, $\inf f$, $|f|_{C^2}$, and $|\pp|_{C^3}$.

    In particular, if $\pp$ satisfies \eqref{condition c_pp}, then the convexity assumption on $\ooo$ can be weakened.
\end{lemma}
\begin{proof}
    For any fixed $T_0\in(0,T)$, we begin by making the necessary assumptions:
    $$ 
    \max\limits_{\pa \ooo\times[0,T_0]}|u_{\nu \nu}|=-\min\limits_{\pa \ooo\times[0,T_0]}u_{\nu \nu}, \quad  N:=-\min\limits_{\pa \ooo\times[0,T_0]} u_{\nu \nu}=-u_{\nu \nu} (z_0,s_0) >0.
    $$ 
    Otherwise, we can easily get from Lemma \ref{lem parabolic upper bound of boundary C^2}. We consider the auxiliary function 
    $$
    P(x,t)=(1+\bb d) \z[Du \cdot (-Dd)-\pp(x,u)\x]+\z(A+\frac{N}{2}\x)h(x), \quad \te{in } \ooo_\mu \times [0,T_0].
    $$
    where $A$ and $\bb$ are two undermined constants.

    Using Lemma \ref{lem h=-d+Kd^2}, there exists a small constant $\frac{1}{10}>\mu>0$, depending on $n,k,\ooo$, and a constant $K>0$ such that $8K\mu\leq 1$ and the function $h(x)=-d(x)+Kd^2(x)$ satisfies 
    \begin{align*}
        \sum_{i,j} \widetilde{F}^{ij}h_{ij} \geq c_1(n,k,\ooo) \sum_{i} \widetilde{F}^{ii}, \quad \te{in } \ooo_\mu\times [0,T).
    \end{align*}
    where $c_1>0$ is defined in Lemma \ref{lem h=-d+Kd^2} and $c_1$ may depend on $c_\ka$.

    We now prove that $P(x,t)\leq 0$ for $x\in \pa\ooo_{\mu}$ or $t=0$. First, note that $P(\cdot,t)|_{\pa \ooo}=0$. For $x\in \pa \ooo_\mu \cap \ooo$, 
    \begin{align*}
        P(x,t) \leq C_1 +\mu \z[\bb C_1-\z(A+\frac{N}{2}\x)(1-K\mu)\x]<0, \quad\te{provided } A> \frac{8}{7}C_1(\bb+\frac{1}{\mu}).
    \end{align*}
    For $x\in \ooo_\mu$, there exists $y\in \pa \ooo$ such that $|x-y|=d(x)$. It follows that
    \begin{align*}
        |Du(x,0)\cdot&(-Dd(x))-\pp(x,u(x,0))|=|Du_0\cdot(-Dd)(x)-\pp(x,u_0(x))|\\
        & =|Du_0\cdot(-Dd)(x)-Du_0\cdot(-Dd)(y)+\pp(y,u_0(y)) -\pp(x,u_0(x))| \\
        &\leq  C_2(\ooo,|u_0|_{C^2},|\pp|_{C^1})|x-y|=C_2 d(x),
    \end{align*}
    and if $A>\frac{8}{7}C_2(1+\bb \mu)$,
    \begin{align*}
        P(x,0) \leq d\z[(1+\bb d)C_2 -\z(A+\frac{N}{2}\x)(1-Kd)\x]<0\quad \te{in }  \ooo_\mu.
    \end{align*}

    Next, we claim that $P|_{\overline{\ooo_\mu}\times[0,T_0]}$ attains its maximum only on $\pa \ooo\times[0,T_0]$. Then 
    \begin{align*}
        0 \leq P_\nu (z_0,s_0) 
        \leq
        \frac{1}{2}\min_{\pa \ooo\times[0,T_0]} u_{\nu \nu}+C_{3}(\ooo,|Du|_{C^0}, |\pp|_{C^1})+A.
    \end{align*}

    To prove this claim, we assume by contradiction that $P|_{\overline{\ooo_\mu}\times [0,T_0]}$ attains its maximum at $(x_0,t_0) \in \Omega_\mu\times(0,T_0]$. By rotating the coordinates, we may assume that $D^2 u(x_0,t_0)$ is diagonal. All subsequent calculations are evaluated at $(x_0,t_0)$. Hence
    \begin{align*}
        0\leq P_t= \z(1+\bb d\x)\z[- u_{jt}d_j-\pp_u u_t\x].
    \end{align*}
    By applying Lemma \ref{lem 2.6}, \eqref{eq elliptic lower, 0=widetilde{P}_{i}}, \eqref{eq elliptic lower, 0 geq widetilde{P}_{ii}} and \eqref{eq parabolic sum geq tilde c_2}, we have the following inequality, analogous to \eqref{eq elliptic upper bound 0 geq I+II+III+IV}:
    \begin{align*}
        0 \geq \widetilde{F}^{ii} P_{ii} -P_t
        \geq& -2\beta {\widetilde{F}^{ii} u_{ii} d_i ^2 } +(1+ \bb d)\z[-d_j \z(\widetilde{F}^{ii}u_{iij}-u_{tj}\x) - 2\widetilde{F}^{ii} u_{ii} d_{ii}\x]   \\
        &-\pp_u(1+\bb d)(\widetilde{F}^{ii}u_{ii}-u_t)+ \z[\z(A + \frac{M} {2}\x)c_1- C_{4}(1+\bb)\x]\sum_{i} \widetilde{F}^{ii}  \\
        \geq& - 2\beta \widetilde{F}^{ii} u_{ii} d_i^2- 2(1+ \beta d) \widetilde{F}^{ii} u_{ii} d_{ii} -\pp_u (1+\bb d)  \widetilde{F}^{ii} u_{ii}  \\
        &+ \z[\z(A + \frac{N}{2}\x)c_1  - C_{5}(1+\bb)\x]\sum_{i} \widetilde{F}^{ii}.
    \end{align*}

    To handle the term $-\sum_i \widetilde{F}^{ii}u_{ii}d_{i}^2$, we partition the indices $i$ into two sets $B$ and $G$. The proof follows the same approach as in Lemma \ref{lem elliptic u_nu nu geq -C}, with $S_k^{ii}$ substituted by $\widetilde{F}^{ii}$.

    Therefore, we can choose sufficiently large constants $A$ and $\bb$ with $A\gg \bb$, ensuring that the maximum of $P|_{\overline{\ooo_\mu}\times[0,T_0]}$ is attained only on $\pa \ooo\times(0,T_0]$. Since $T_0\in(0,T)$ is arbitrary, the proof is complete.
\end{proof}

\begin{lemma}\label{lem parabolic upper bound of boundary C^2}
    Under the assumptions of Theorem $\ref{thm parabolic C^2}$, we have 
    \begin{equation}\label{eq parabolic u_nu nu leq C}
        \sup_{\pa \ooo\times[0,T)}u_{\nu\nu} \leq N_7,
    \end{equation}
    where $N_7>0$ depends on $n$, $k$, $\ooo$, $\tilde{c}_0$, $\tilde{c}_1$, $N_1$, $N_5$, $|u_0|_{C^2}$, $|Du|_{C^0}$, $\inf f$, $|f|_{C^2}$, and $|\pp|_{C^3}$.

    In particular, if $\pp$ satisfies \eqref{condition c_pp}, then the convexity assumption on $\ooo$ can be weakened.
\end{lemma}
\begin{proof}
    The proof of this lemma can be established by combining the arguments of Lemma \ref{lem elliptic u_nu nu leq C} and Lemma \ref{lem parabolic lower bound of boundary C^2}. Here, we only provide a sketch of the main ideas. 
    
    \textbf{Step 1.} For any fixed $T_0\in (0,T)$, we begin by making the necessary assumptions: 
    $$
    \max_{\pa \ooo\times[0,T)} |u_{\nu\nu}|=\max_{\pa \ooo\times [0,T)} u_{\nu\nu}, \quad N:=\max\limits_{\pa \ooo\times[0,T_0]} u_{\nu \nu}=u_{\nu \nu} (\tilde{z}_0,\tilde{s}_0)>0.    
    $$
    Define $\widetilde{P}(x) = (1+\beta d)\z[Du \cdot(-D d)-\pp(x,u) \x] - \z(A +\frac{N}{2}\x)h(x)$, in $\ooo_\mu\times[0,T_0]$, 
    where $h$ is defined in Lemma \ref{lem h=-d+Kd^2}, $\bb$ and $A$ are positive constants to be determined.

    \textbf{Step 2.} Following a similar argument as in Lemma \ref{lem parabolic lower bound of boundary C^2}, by choosing $A$ sufficiently large such that 
    $\widetilde{P}=0$ on $\pa \ooo$, and $\widetilde{P}(x,t)> 0$ on $(\pa \ooo_\mu \cap \ooo)\times[0,T_0]$ or $\{t=0\}$.

    \textbf{Step 3.} 
    We assume that $\widetilde{P}|_{\overline{\ooo_\mu}\times[0,T_0]}$ attains its minimum at $(\tilde{x}_0,\tilde{t}_0) \in \Omega_\mu\times(0,T_0]$ and $D^2 u(\tilde{x}_0,\tilde{t}_0)$ is diagonal. Similar to the calculation in Lemma \ref{lem parabolic lower bound of boundary C^2}, we have
    \begin{align*}
        0 \leq  \widetilde{F}^{ii} \widetilde{P}_{ii}-P_t  
        \leq&  -2\bb \widetilde{F}^{ii} u_{ii} d_i^2-2(1+ \beta d) \widetilde{F}^{ii} u_{ii} d_{ii}-\pp_u(1+\bb d)  \widetilde{F}^{ii}u_{ii}   \\
        &+ \z[-\z(A+\frac{N}{2}\x)c_1+ C (1+\bb)\x]\sum_{i} \widetilde{F}^{ii}.
    \end{align*}

    The subsequent proof follows the argument as in Lemma \ref{lem elliptic u_nu nu leq C}, with $S_k^{ii}$ substituted by $\widetilde{F}^{ii}$. Therefore, we can choose sufficiently large $A$ and $\bb$ with $A\gg \bb$, ensuring that the minimum of $\widetilde{P}\big|_{\overline{\ooo_\mu}\times[0,T_0]}$ is attained only on $\pa \ooo\times(0,T_0]$. Since $T_0\in(0,T)$ is arbitrary, the proof is complete. 
\end{proof}

\begin{proof}[Proof of Theorem $\ref{thm parabolic C^2}$]
    Combining Lemmas \ref{lem parabolic global C^2 reduce to boundary}, \ref{lem parabolic lower bound of boundary C^2}, and \ref{lem parabolic upper bound of boundary C^2}, we complete the proof of Theorem \ref{thm parabolic C^2}.
\end{proof}

\subsection{Proof of Theorem \ref{thm 1.5-parabolic}}

Based on the preceding analysis, we obtain the $u_t$-estimate and a priori $C^2$ estimates for the parabolic Neumann problem \eqref{eq parabolic sum hessian-1}. These estimates imply that \eqref{eq parabolic sum hessian-1} is uniformly parabolic. The concavity of $\widetilde{F}=\log S_k$ allows us to apply the arguments of \cite{Lieberman-Book-1996} and \cite{Lieberman-Trudinger-1986-TAMS}, yielding global $C^{2,\ga}$-estimates for $u$. Combining the method of continuity \cite[Corollary 14.9]{Lieberman-Book-1996}
, we conclude that a smooth solution $u(x,t)$ to \eqref{eq parabolic sum hessian-1} exists for all $t\geq 0$. 

By \cite[Lemma 10.1]{SK-poincare-2003}, the solution $u(x,t)$ converges smoothly as $t\to \wq$ to a limit $u^\wq$, which solves the Neumann problem \eqref{eq convergent elliptic Equation}. Moreover, if \eqref{condition c_f} holds, the decay estimate \eqref{eq parabolic u_t e^{r t}} implies $|u_t(x,t)|\leq Ce^{-c_ft}$, yielding exponential convergence. The uniqueness follows directly from the maximum principle and Hopf's lemma.

\section{Proof of Theorem \ref{thm 1.6-classical parabolic}}\label{sec-classical parabolic}
\setcounter{equation}{0}

In this section, we prove Theorem \ref{thm 1.6-classical parabolic}, adapting techniques from \cite{chen-ma-zhang-2021,Ma-Wang-Wei-2018-JFA,Schnurer-Schwetlick-PJM-2004}.

\subsection{An elliptic Neumann problem}

We begin by considering the following elliptic Neumann problem, which plays a crucial role in the proof of Theorem \ref{thm 1.6-classical parabolic}.

\begin{theorem}\label{thm 7.1}
    Let $\ooo\subset \rr^n$ be a smooth uniformly convex domain. Suppose $u_0$ is defined in Theorem $\ref{thm 1.6-classical parabolic}$ and $f,\pp \in C^{\wq}(\overline{\ooo})$ satisfies $f>0$. Then there exists a unique $s\in \rr$ and a $(k-1)$-admissible function $u \in C^\wq(\overline{\ooo})$ solving the Neumann problem
    \begin{equation}\label{eq 7.1}
        \left\{\begin{aligned}
            & S_k(D^2 u) =f(x)e^s && \te{in } \ooo, \\
            & u_\nu=\pp(x) && \te{on }\pa \ooo.
        \end{aligned}\right.
    \end{equation}
    Moreover, the solution $u_{\operatorname{ell}}^\wq$ is unique up to a constant.
\end{theorem}
\begin{proof}
    For any fixed $\es\in (0,1)$, we first prove the existence of a unique $(k-1)$-admissible solution $u_{\es,s}$ to the following perturbed equation:
    \begin{equation}\label{eq *}\tag{$*_{\es,s}$}
        \left\{\begin{aligned} 
            & S_k(D^2 u) = f(x)e^{s+\es u} && \te{in } \Omega,  \\
            & u_\nu =\pp(x) && \te{in } \partial \Omega.  \\
         \end{aligned}\right.
    \end{equation}
    Indeed, $u_{\es,s}(x)=u_{\es,0}(x)-\frac{s}{\es}$, which implies $u_{\es,s}$ is strictly decreasing in $s$.
    
    In the following, we will prove that for each $\es\in (0,1)$, there exists a unique, uniformly bounded constant $s_\es$ such that $|u_{\es,s_\es}|_{C^k(\overline{\ooo})}$ is uniformly bounded. These uniform bounds enable us to extract a subsequence converging to a solution of  \eqref{eq 7.1}.
    
    \textbf{Step 1:} For sufficiently large $M$, we claim that 
    $$ 
    u_\es^+ = u_0+\frac{M}{\es}\quad \te{and} \quad u_\es^-=u_0-\frac{M}{\es}
    $$ 
    are a supersolution and subsolution of \eqref{eq *}, respectively. That is, $u_\es^- \leq u_{\es, 0}\leq u_\es^+$. 
    
    Indeed, by choosing $M=|u_0|+|\log \inf f|+|\log S_k(D^2 u_0)|+1$, we have
    \begin{align*}\label{eq a^{ij}-c-I}
        a^{ij}(u_{\es ,0}-u_\es^+)_{ij}-c(x)(u_{\es ,0}-u_\es^+) 
        = &S_k(D^2 u_{\es ,0} )e^{-\es u_{\es ,0} } -S_k(D^2 u_\es^+)e^{-\es u_\es^+} \notag \\
        =& f(x) -S_k (D^2 u_0)e^{-M-\es u_0} >0,
    \end{align*}
    where $a^{ij}=\int_0^1 S_k^{ij}({D^2 u^t})e^{-\es u^t} dt$ is positive definite, $c(x)=\es \int_0^1 S_k(D^2 u^t) e^{-\es u^t} dt>0$ and $u^t:=tu_{\es,0}+(1-t) u_\es^+$. Additionally, $(u_{\es,0}-u_\es^+)_\nu=0$ on $\pa \ooo$.

    The maximum principle and Hopf's lemma imply that $u_{\es ,0}<u_\es^+$ on $\overline{\ooo}$. Similarly, $u_{\es,0}>u_\es^-$ on $\overline{\ooo}$. Thus we have $u_{\es, M}< u_0< u_{\es, -M}$ on $\overline{\ooo}$.  
    By strictly decreasing property of $u_{\es, s}$ in $s$, for each $\es\in(0,1)$, we can find a unique $s_\es\in (-M, M)$ such that $u_{\es, s_{\es}}(y_0) =u_0(y_0)$ for a fixed $y_0\in \ooo$. In particular, $|\es u_{\es,s_{\es}}|_{C^0(\overline\Omega)} \leq 2M + \es |u_0|_{C^0(\overline\Omega)} \leq 3M$.

    \textbf{Step 2:} 
    We prove that for any $\es\in (0,1)$, we have $|Du_{\es,s_{\es}}|_{C^0(\overline{\ooo})}\leq C_0$, 
    where $C_0$ is a positive constant independent of $\es$ and $|u_{\es,s_\es}|_{C^0}$. 
    We continue to adopt the notation for $\widetilde{F}$ introduced in Section \ref{sec-parabolic}. Since $\ooo$ is uniformly convex, there exists a defining function $h$ for $\ooo$ satisfying 
    \begin{equation}\label{eq defining function}
        \left\{\begin{aligned}
            & h=0, \quad Dh=-Dd=\nu \quad \te{on } \pa \ooo, \\
            & h<0,\quad|h|_{C^3} \leq C(n,\ooo), \quad 0<\tilde{\ka}_1 I\leq D^2 h\leq \tilde{\ka}_2 I \quad \te{in } \ooo.
        \end{aligned}\right.
    \end{equation}

    Let $\al_0>0$, and define $w=u-\pp h$ (where $u=u_{\es,s_\es}$) and the auxiliary function
    \begin{align*}
        G = \log|Dw|^2 +\al_0h.
    \end{align*}
    Suppose $G$ attains its maximum at $x_0\in \overline{\ooo}$. The proof of the global gradient estimate is divided into two cases, with all subsequent calculations evaluated at $x_0$.

    \noindent\textbf{Case 1.} $x_0\in \ooo$. Without loss of generality, we may assume that at the point $x_0$, $|Du|$ is sufficiently large and $D^2 u$ is diagonal. Hence $|Dw|$ is also large and comparable to $|Du|$. 
    
    Note that $|\es u_{\es,s_\es}|_{C^0}\leq 3M$. Using Lemma \ref{lem 2.3}, we have 
    \begin{align}\label{eq sum_i widetilde{F}^{ii} geq tilde{c}_3>0}
        \sum_i \widetilde{F}^{ii}(D^2 u_{\es,s_\es})\geq \frac{c_0(n,k,\inf S_k)}{S_k(D^2 u_{\es,s_\es})}:=\tilde{c}_3(n,k,M,\inf f,|f|_{C^0})>0.
    \end{align}
    By direct calculation, we derive $0=G_i=\frac{2w_l w_{li}}{|Dw|^2}+\al_0 h_i$ and 
    \begin{align}
        0\geq \widetilde{F}^{ii}G_{ii} 
        = & 2\widetilde{F}^{ii}\frac{w_{ii}^2}{|Dw|^2}+2\widetilde{F}^{ii}\frac{w_l w_{lii}}{|Dw|^2}-\al_0^2 \widetilde{F}^{ii}h_i^2 +\al_0 \widetilde{F}^{ii}h_{ii} \nonumber\\
        \geq & \al_0(\tilde{\ka}_1-\al_0 C_1 )\sum_i \widetilde{F}^{ii}+2\widetilde{F}^{ii}\frac{w_l w_{lii}}{|Dw|^2}.\label{eq thm 7.1 0 geq widetilde{F}^{ii}G_{ii}-first}
    \end{align}
    We now deal with the last term. Observe that
    \begin{align*}
        w_l \widetilde{F}^{ii}w_{lii}=& [u_l-(\pp h)_l]\widetilde{F}^{ii}[u_{lii}-(\pp h)_{lii}] 
        \geq  -C_2|Du|-C_2 |Du|\sum_i \widetilde{F}^{ii}.
    \end{align*}
    By choosing $\al_0>0$ sufficiently small and using \eqref{eq sum_i widetilde{F}^{ii} geq tilde{c}_3>0}, inequality \eqref{eq thm 7.1 0 geq widetilde{F}^{ii}G_{ii}-first} simplifies to
    \begin{align*}
        0\geq \widetilde{F}^{ii}G_{ii} 
        \geq & \z[ \al_0\z(\tilde{\ka}_1-\al_0 C_1\x)-\frac{C_3}{|Dw|}\x]\sum_i \widetilde{F}^{ii} \geq \frac{\tilde{c}_3 \al_0 \tilde{\ka}_1}{2}-\frac{C_4}{|Dw|}.
    \end{align*}
    Therefore, we establish a bound for $|Dw(x_0)|$, which implies $\sup_{\overline{\ooo}}|Du_{\es,s_{\es}}|\leq C_5$.

    \noindent\textbf{Case 2.} $x_0 \in \pa\ooo$. Then
    \begin{align}\label{eq thm 7.1 0 leq G_nu}
        0\leq G_\nu=G_i \nu^i=\frac{2w_l w_{li}}{|Dw|^2}\nu^i+\al_0 D_\nu h =\frac{2w_l w_{li}}{|Dw|^2}\nu^i+\al_0.
    \end{align}
    Differentiating the boundary condition $u_i \nu^i=u_\nu(x)=\pp(x)$, we have $u_{il}\nu^i +u_i {\nu^i}_l=\pp_l$. On the other hand, since $w=u-\pp h$, at the point $x_0\in \pa\ooo$, we have
    \begin{align*}
        w_l& =u_l-(\pp h)_l=u_l -\pp h_l, \\
        w_{li}&=u_{li}-(\pp h)_{li}
        =u_{li} -\pp_l h_i -\pp_i h_l -\pp h_{li},\\
        w_{li} \nu^i  & 
        = (u_{li}\nu^i-\pp_l)-\pp_\nu h_l -\pp h_{li}\nu^i  
        =-u_i {\nu^i}_l-\pp_\nu h_l -\pp h_{li}\nu^i.
    \end{align*}
    Bt the uniformly convexity of $\ooo$, it follows that
    \begin{align*}
        w_l w_{li}\nu^i=& -u_i u_l {\nu^i}_l -u_l(\pp_\nu h_l +\pp h_{li}\nu^i)+\pp h_l(u_i {\nu^i}_l+\pp_\nu h_l+\pp h_{li}\nu^i)\\
        \leq & -\ka_{\operatorname{min}} (|Du|^2-|D_\nu u|^2) + C_6(|Du|+1)\leq -\ka_{\operatorname{min}} |Du|^2+C_7(|Du|+1),
    \end{align*}
    where $\ka_{\operatorname{min}}$ is the minimum principal curvature of $\pa \ooo$.
    Consequently, \eqref{eq thm 7.1 0 leq G_nu} becomes
    \begin{align*}
        0\leq \al_0 -\frac{\ka_{\operatorname{min}}}{2}+\frac{C_8}{|Dw|}.
    \end{align*}
    Choosing $\al_0=\frac{\ka_{\operatorname{min}}}{4}$, we deduce that $|Du_{\es,s_\es}|\leq C_9$. This establishes the global gradient estimate for $u_{\es,s_\es}$, and the bound is independent of $\es$, and $|u_{\es,s_\es}|_{C^0}$.
    
    \textbf{Step 3:} We now prove the uniform $C^0$ bounds for $u_{\es,s_\es}$. By construction, $u_{\es, s_\es}(y_0)=u_0(y_0)$, and thus for any $x\in \ooo$,
    \begin{align*}
        |u_{\es, s_\es} (x)|\leq  |u_{\es, s_\es}(y_0)| + |Du_{\es, s_\es}|_{C^0}|x-y_0| \leq C_{10}.
    \end{align*}

    Moreover, the arguments in Section \ref{sec 4-second derivative} yield a global second-order derivatives estimate $|D^2 u_{\es,s_\es}|\leq C_{11}$. A crucial observation in the proof of Lemma \ref{lem elliptic global C^2 reduce to boundary} is that the uniform convexity of the domain yields $\pp_u-2 {\nu^1}_1\leq -2\ka_{\operatorname{min}}<0$.

    By passing to a subsequence (still denoted by $\es \to 0$), we obtain $s_{\es} \to s^\infty$ and $u_{\es, s_{\es}}$ converges to a $(k-1)$-convex function $u_{\operatorname{ell}}^\infty$ that satisfies equation \eqref{eq 7.1} with $s = s^\infty$. The higher-order regularity follows from the estimates in \cite{Lieberman-Trudinger-1986-TAMS}, while the uniqueness of $s^\wq$ is a consequence of the maximum principle and Hopf's lemma.
\end{proof}

\subsection{Proof of Theorem \ref{thm 1.6-classical parabolic}}

To prove Theorem \ref{thm 1.6-classical parabolic}, we need to establish the a priori estimates for equation \eqref{eq parabolic sum hessian-2}.

\begin{proof}
    Assume $u$ is the $(k-1)$-admissible solution of \eqref{eq parabolic sum hessian-2}.

    $(1)$ $u_t$-estimate. By Lemma \ref{lem parabolic u_t estimate} $(1)$, we obtain
    $$
    |u_{t}| \leq  |u_{t}(\cdot,0)|_{C^0(\overline{\ooo})}\leq C_1(n,k,|u_0|_{C^0},\inf f,|f|_{C^0}), \quad\te{ in } \overline{\ooo}\times[0,T).
    $$

    $(2)$ $|Du|$-estimate. For any $T_0\in (0,T)$. We consider the auxiliary function
    $$
    G(x,t) = \log |Dw|^2 +\al_0 h, \quad \te{in } \overline{\ooo} \times [0,T_0],
    $$
    where $\al_0>0$, $w=u-\pp d$ and $h$ is the defining function for $\ooo$ satisfying \eqref{eq defining function}. 
    
    Suppose $G$ attains its maximum over $\overline{\ooo}\times[0,T_0]$ at $(x_0,t_0)$. If $t_0=0$, the global gradient estimate follows directly. Now assume  $t_0>0$, we divide the proof into two cases, with all subsequent calculations are evaluated at $(x_0, t_0)$.

    \textbf{Case 1.} $x_0\in \ooo$. Without loss of generality, we may assume at $(x_0,t_0)$, $|Du|$ is sufficiently large and $D^2 u$ is diagonal. Then $|Dw|$ is also large and comparable to $|Du|$ at $(x_0,t_0)$. Since $S_k(D^2 u) = f(x) e^{u_t}$, by Lemma 3.3, we have
    $$
    \sum_i F^{ii} \geq \frac{c_0(n,k,\inf S_k)}{S_k} = \tilde{c}_4(n,k,|u_0|_{C^0}, \inf f, |f|_{C^0} ) > 0.
    $$

    Through direct calculation, at the point $(x_0, t_0)$, we have
    \begin{align*}
        0 = G_i = 2\frac{ w_l w_{li}}{|Dw|^2} + \al_0 h_i, \quad 0\leq G_t =2 \frac{w_l w_{lt}}{|Dw|^2},
    \end{align*}
    and
    \begin{align*}
        0 \geq& F^{ii} G_{ii}-G_t =  \frac{2}{|Dw|^2} w_l (\widetilde{F}^{ii}w_{lii} -w_{lt}) + \frac{2}{|Dw|^2} \widetilde{F}^{ii} w_{li}^2 + \al_0 F^{ii} h_{ii} - \al_0^2 \widetilde{F}^{ii} h_i^2 \\
        \geq & \al_0(\tilde{\ka}_1-\al_0 C_2) \sum_i \widetilde{F}^{ii} + \frac{2}{|Dw|^2} [u_l-(\pp h)_l][(\widetilde{F}^{ii}u_{lii}-u_{lt})-\widetilde{F}^{ii}(\pp h)_{lii}]\\
        \geq& \frac{\al_0 }{2}\ka_1 \tilde{c}_4-\frac{C_3}{|Dw|}. 
    \end{align*}
    where $\al_0$ is a sufficiently small constant. Thus, $|Dw|\leq C_4$, which implies that $|Du(x,t)|\leq C_5$ for any $(x,t)\in \overline{\ooo}\times[0,T_0]$. Moreover, $C_5$ is independent of $|u|_{C^0}$ and $T_0$.

    \textbf{Case 2.} $x_0\in \pa \ooo$. Following the arguments in Theorem \ref{thm 7.1}, we obtain $|Du| \leq C_6$.

    $(3)$ $|D^2 u|$-estimate. Following the arguments in Section \ref{sec 4-second derivative}, we derive $\sup_{\overline{\ooo}\times[0,T)}|D^2 u|\leq C_7$, where $C_7$ is independent of $|u|_{C^0}$.

    $(4)$ $C^0$-estimate. Let $u^\wq(x,t)=u_{\operatorname{ell}}^\wq(x)+s^\wq t$ and $w=u-u^\wq$, where $(s^\wq,u_{\operatorname{ell}}^\wq)$ is a solution to \eqref{eq 7.1}. Then $u^\wq$ and $w$ satisfies 
    \begin{equation*}
        \left\{\begin{aligned}
            & u^\wq_t=\log S_k(D^2 u^\wq)-\log f(x) && (x,t)\in \ooo\times\rr,\\
            & u^\wq_\nu=\pp(x) && (x,t)\in \pa \ooo \times\rr, \\
            & u^\wq(x,0)=u_{\operatorname{ell}}^\wq(x) && x\in \ooo,
        \end{aligned}\right. 
    \end{equation*} 
    and 
    \begin{equation*}
        \left\{\begin{aligned}
            & w_t=a^{ij}w_{ij} && (x,t)\in \ooo\times\rr,\\
            & w_\nu=0 && (x,t)\in \pa \ooo \times\rr, \\
            & w(x,0)=u_0(x)-u_{\operatorname{ell}}^\wq(x) && x\in \ooo,
        \end{aligned}\right. 
    \end{equation*} 
    where $a^{ij}=\int_{0}^1 S_k^{ij}(tD^2u+(1-t)D^2 u^\wq)dt$ is positive definite. 
    By the parabolic strong maximum principle and Hopf's lemma, the 
    extrema of $w$ on $\overline{\ooo}\times[0,\wq)$ must occur at $t=0$. 
    Therefore, 
    $$
    \inf_{\ooo} w(\cdot,0)<w(x,t)<\sup_{\ooo} w(\cdot,0), \quad \te{for any } (x,t) \in \ooo \times(0,\wq).
    $$
    That is, $s^\wq t-|u_0|_{C^0}-2|u_{\operatorname{ell}}^\wq|_{C^0}\leq u \leq s^\wq t+|u_0|_{C^0}+2|u_{\operatorname{ell}}^\wq|_{C^0}$ in $\overline{\ooo}\times[0,\wq)$.

    $(5)$ Smoothness, Long-time existence, and Convergence.
    
    The remainder of the proof follows the same arguments as in Subsection 8.3 of \cite{chen-ma-zhang-2021}. Consequently, we establish the smoothness of $u$ (with uniform bounds on all higher-order derivatives) and its long-time existence. As $t \to \wq$, the solution $u$ converges smoothly to a translating solution. Specifically, there exists a constant $a^\wq$, depending on $u_{\operatorname{ell}}^\wq$, such that
    \begin{equation*}
        \lim_{t \to \wq} |u(x,t)-u_{\operatorname{ell}}^\wq(x)-s^\wq t-a^\wq|_{C^l(\overline{\ooo})}=0, \quad\te{for any integer } l\geq 0,
    \end{equation*}
    where the pair $(s^\wq,u_{\operatorname{ell}}^\wq+a^\wq)$ also solves \eqref{eq 7.1}. 

    $(6)$ Uniqueness follows from the parabolic maximum principle and Hopf's lemma.
\end{proof}

The authors would like to thank Professor Xi-Nan Ma for the constant encouragement in this subject.

\section*{Conflict of interest}

The authors state no conflict of interest.

\section*{Data availability statement}

The data used to support the findings of this study are included within the article.

\end{document}